\documentclass{article}

\def\and{\quad\text{and}\quad}

\def\eps{\varepsilon}
\def\vphi{\varphi}
\def\d{\textup{d}}
\def\Rtn{\textup{Rtn}}

\def\and{\quad\text{and}\quad}

\def\leqs{\leqslant}
\def\geqs{\geqslant}
\def\back{\backslash}

\usepackage[pdftex]{graphicx}
\usepackage{subfigure}
\usepackage{amsfonts,amssymb,amsmath,amsthm}

\newtheorem{theo}{Theorem}[section]
\newtheorem{prop}{Proposition}[section]
\newtheorem{lemma}{Lemma}[section]

\begin{document}
\title{Polynomial decay of correlations in linked-twist maps}
\author{J. Springham, R. Sturman\\
Department of Applied Mathematics, University of Leeds}
\maketitle
\begin{abstract}
Linked-twist maps are area-preserving, piece-wise diffeomorphisms,  defined on a subset of the torus. They are non-uniformly hyperbolic generalisations of the well-known Arnold Cat Map. We show that a class of canonical examples have polynomial decay of correlations for $\alpha$-H\"{o}lder observables, of order $1/n$. 
\end{abstract}

\section{Introduction}\label{Intro}

A common method of classifying the complicated statistical properties of a dynamical system is to establish its rate of decay of correlations. This is a measure of the rate at which the system mixes up initial conditions, independently of how this mixture is measured. For example, correlations for uniformly expanding maps on an interval can be easily shown to decay at exponential rate. The exponential nature of the decay stems, of course, from the exponential divergence of nearby initial conditions intrinsic to chaotic dynamics. Many examples in one dimension are now well-known, and particular interest has been shown to cases in which periodic boundary conditions are replaced with an artefact designed to destroy the uniformity of the chaos, and hence slow the rate of mixing.

Similar results in two dimensions are also established. For example, the Arnold Cat Map (and indeed any hyperbolic toral automorphism) can be shown to be exponentially mixing by appealing to the linearity of the map and using Fourier series \cite{baladi2000positive}. This fast mixing rate has also been shown to be slowed by the introduction of a carefully chosen perturbation near the fixed point at the origin\cite{artuso1999correlation}. However, on the whole, interesting behaviour designed to slow down mixing rates tends to be restricted to behaviour at isolated points.

In this paper we consider a linked-twist map, which could be viewed as a non-uniformly hyperbolic version of the uniformly hyperbolic Cat Map. It is Lebesgue measure-preserving, and is defined on a two dimensional manifold with non-trivial boundary. As such it is an instructive map, in that it reveals transparently both the source of its hyperbolicity, and the manner in which the uniformity of hyperbolicity is lost. 

The understanding of the dynamical properties of such maps was instigated by \cite{d1}, who showed that they were almost Anosov and \cite{be}, who demonstrated ergodicity for a related (nonlinear) map. This was soon enhanced by \cite{woj} and \cite{p1}, who proved mixing and the Bernoulli property respectively for families of linear linked twist maps. The former also treated similar examples defined on linked circular annuli on the plane. That these are mixing was conjectured by \cite{woj}, the geometrical argument to demonstrate this being completed by \cite{springham}. (See also \cite{springham2008ergodic}.)

At this stage the theoretical development of such maps was left (with the exception of some exotic variations due to \cite{nicol2,nicol1}). However, in recent years \cite{ow2,sturman} showed that this class of maps underpins a wide variety of fluid mixing devices. In this context, the existence of the boundary is crucial, as it can be used for the first time to make rigorous statements about physically realizable phenomena (described in, for example \cite{gouillart1994walls,gouillart2008slow}) in practical applications to model the effect of hydrodynamical boundary conditions in experimental devices \cite{ss1}. For this reason the specific details of the dynamical mechanism underlying the mixing properties of this particular system are likely to be of wider interest.

In this paper however, we are concerned purely with the dynamical behaviour of the non-uniformly hyperbolic piece-wise diffeomorphism with boundary. We note that this is not the only example of a non-uniformly hyperbolic generalization of the Arnold Cat map. \cite{cerbelli2005continuous} introduced another such map, also studied by \cite{mackay06}, in which non-uniformity stems from a non-monotonic twist function. In that case however, there exists a Markov partition which allows much immediate analysis. In a linked twist map the dynamics are arguably more intricate, since a Markov partition does not exist.

In the following we are concerned with a map on the two-dimensional torus $\mathbb{T}^2=\mathbb{S}^1\times\mathbb{S}^1$. Rather than use the more standard unit interval we denote $\mathbb{S}^1=[0,2]$, with opposite ends identified (this is because we will largely be concerned with a subset of the torus that can now be denoted $[0,1] \times [0,1]$). Let $(x,y)\in\mathbb{S}^1\times\mathbb{S}^1$ give coordinates on $\mathbb{T}^2$. We define annuli
$$
P=\mathbb{S}^1\times\left[0,1\right]\subset\mathbb{T}^2 \mbox{ and }
Q=\left[0,1\right]\times\mathbb{S}^1\subset\mathbb{T}^2.
$$
We will use the notation $R=P\cup Q$ and $S=P\cap Q$. Define \emph{twist maps} $F:P\to P$ and $G:Q\to Q$ by
$$
F(x,y)=(x+2y,y) \mbox{ and }
G(x,y)=(x,y+2x).
$$
Note that $F$ and $G$ leave invariant the boundaries of $P$ and $Q$ respectively. Let $F=\textup{id}$ (the identity map) on $R\backslash P$ and $G=\textup{id}$ on $R\backslash Q$ so that both $F$ and $G$ are both continuous and moreover preserve the Lebesgue measure $\mu$ on $R$. Their composition, the \emph{linked-twist map} $H=G\circ F$, is illustrated in Figure~\ref{fig:torus}. It is a Lebesgue measure-preserving piece-wise diffeomorphism of $R$ into itself.

\begin{figure}[htp]
\centering
\subfigure[$P$, $Q$ and $S$]{\includegraphics[totalheight=0.18\textheight]{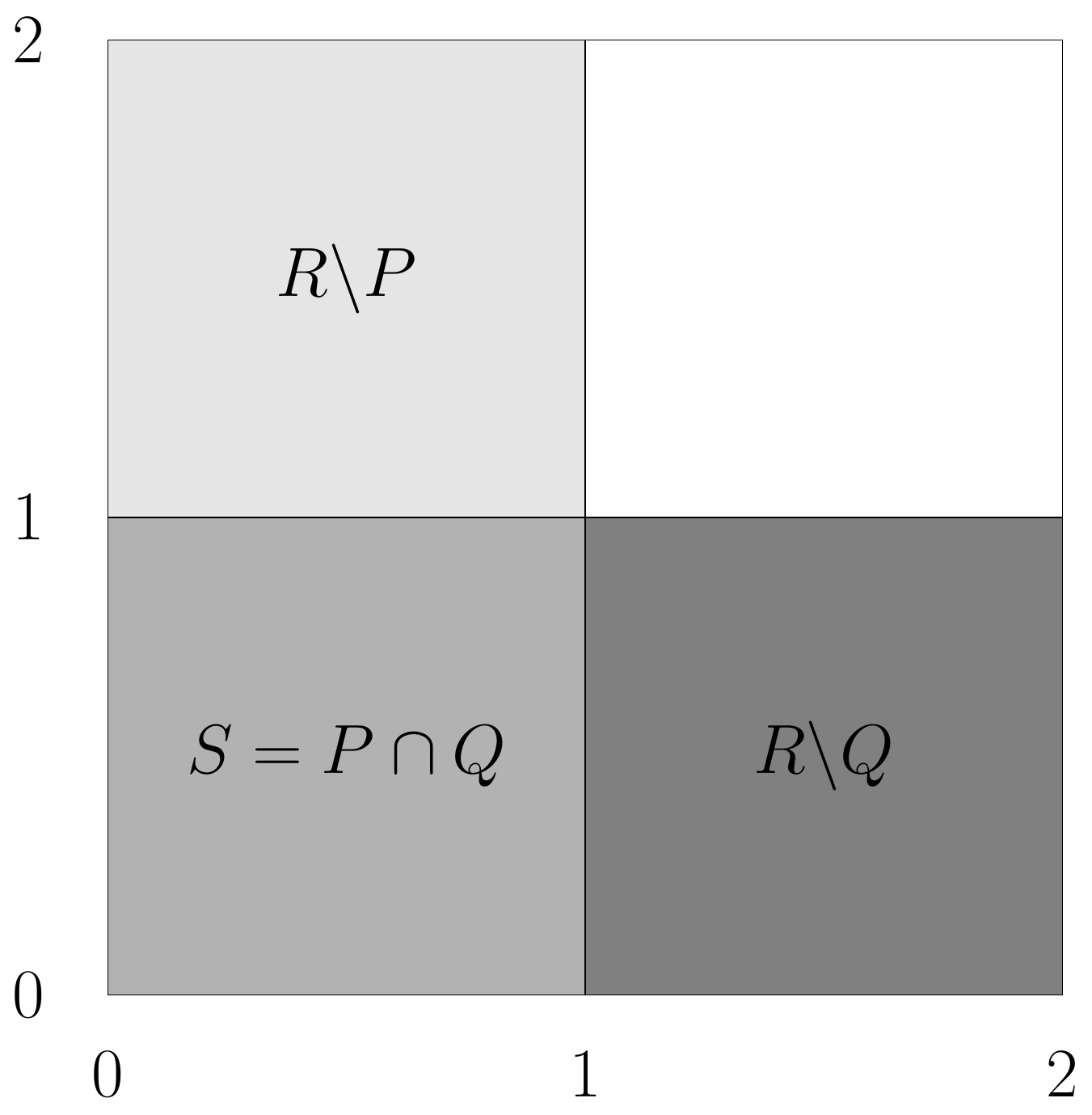}\label{fig:torus1}}\hfill
\subfigure[Image under $F$]{\includegraphics[totalheight=0.18\textheight]{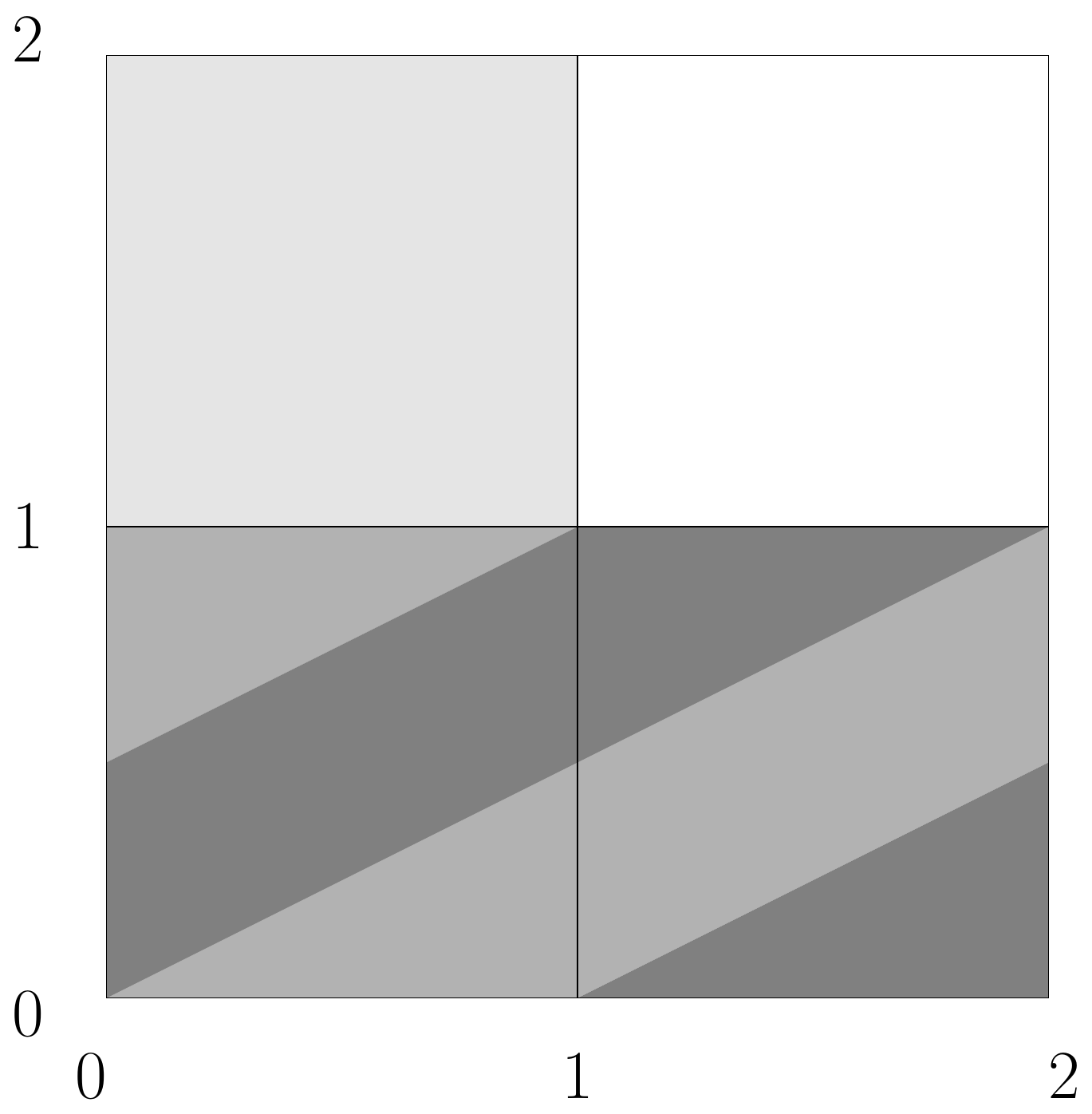}\label{fig:torus2}}\hfill
\subfigure[Image under $H$]{\includegraphics[totalheight=0.18\textheight]{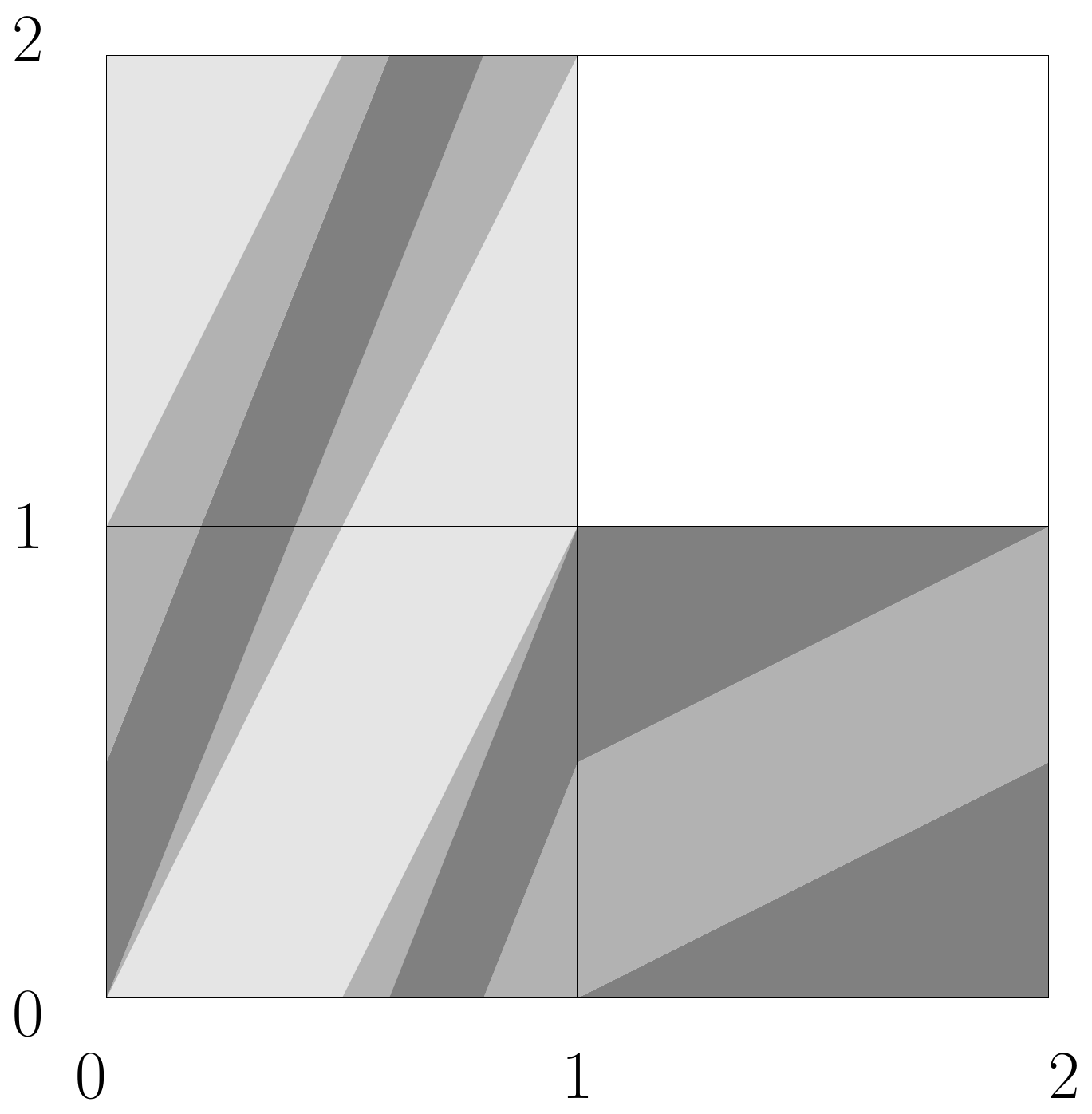}\label{fig:torus3}}\hfill
\caption[A toral linked-twist map]{Illustration of the toral linked-twist map $H:R\to R$. The white region shows $\mathbb{T}^2\back R$. In (a) the region $S$ is coloured mid-grey, with the remainder of $P$ shown in dark-grey and the remainder of $Q$ shown in light-grey. Figure (b) shows the image of these sets under the twist $F$ whilst (c) shows the image under $H=G\circ F$.}
\label{fig:torus}
\end{figure}

\cite{woj} showed that $H$ has the $K$-property and the result of \cite{ch} shows that in fact it is Bernoulli. It is non-uniformly hyperbolic, the proof of non-zero Lyapunov exponents $\mu$-a.e.\ following from an invariant cone field; for details see \cite{sturman}. No results concerning the rate of mixing for $H$ are known to us.

For $n\in\mathbb{N}$ and for any pair of bounded, measurable functions $\vphi,\psi:R\to\mathbb{R}$ (`observables') define the \emph{correlation function}
\begin{equation}\label{eqn:corr_fn}
C_n(\vphi,\psi,H,\mu)=\int_R(\vphi\circ H^n)\psi\d\mu-\int_R\vphi\d\mu\int_R\psi\d\mu.
\end{equation}
It is well-known that $(H,\mu)$ is mixing if and only if $C_n\to 0$ for any such pair of observables.

The rate of \emph{decay of correlations for $H$} refers to the order of this convergence for sufficiently regular $\vphi$ and $\psi$. Let $\mathcal{H}_{\alpha}$ denote the space of real-valued, $\alpha$-H\"{o}lder functions on $R$. These are the functions $\vphi:R\to\mathbb{R}$ for which there are positive constants $\alpha$ and $C$ so that for all $z,z'\in R$ sufficiently close
$$
|\vphi(z)-\vphi(z')|\leqs C\d(z,z')^{\alpha},
$$
where $\d(\cdot,\cdot)$ denotes distance on $R$. As is common we make the further assumption $\int_R\psi\d\mu=0$, which simplifies \eqref{eqn:corr_fn} at no expense of generality.

The main result of our paper is the following.

\begin{theo}\label{thm:main}
If $\vphi,\psi\in\mathcal{H}_{\alpha}$ then $|C_n(\vphi,\psi,H,\mu)|=\mathcal{O}(1/n)$.
\end{theo}

It is important to remark now that, although both statement and proof of Theorem~\ref{thm:main} make explicit use of the particular annuli $P$ and $Q$ defined above, this restriction is little more than a notational convenience. All of our results hold, with only superficial alterations, in the general case $P=\mathbb{S}^1\times[p_0,p_1]$, $Q=[q_0,q_1]\times\mathbb{S}^1$, for any choice of $p_0\neq p_1, q_0\neq q_1\in\mathbb{S}^1$, with $F$ and $G$ appropriately re-defined also. We remark on this further at the end of Section~\ref{Proof}.

We note, given that linked twist maps can be used as a model for a wide variety of mixing devices \cite{sturman}, that Theorem~\ref{thm:main} gives a practical bound on mixing rates in such applications. Moreover, it can be shown that in general, the polynomial rate given is indeed attained by typical observables. This argument and its relevance in applications is discussed in detail in~\cite{ss1}.

Our paper is organised as follows. Our proof of Theorem~\ref{thm:main} uses certain recent results from the dynamical billiards literature and we give a synopsis of these in Section~\ref{Background}. The results essentially reduce the problem to a detailed analysis of an induced map given by first returns to $S\subset R$. This is carried out in Sections~\ref{Partition} through~\ref{OneStep}; in particular, in Section~\ref{Partition} we study the partition of $S$ induced by the return map, in Section~\ref{Bernoulli} we show that the return map is Bernoulli and in Section~\ref{OneStep} we show that a technical condition regarding \emph{local expansion factors} (to be defined in Section~\ref{Background}) is satisfied. In Section~\ref{Proof} we bring these results together to conclude the proof of Theorem~\ref{thm:main}. Finally in Section~\ref{Others} we collect a few thoughts regarding potential extensions and generalisations of our result.

\section{Decay of correlations in hyperbolic systems}\label{Background}

We describe some recent results concerning the decay of correlations in systems with some hyperbolicity, the foundations of which are to be found in two seminal papers of Young \cite{young,young2}.

Let $X$ be a Riemannian manifold, possibly with boundary, and let $T:X\to X$ be a hyperbolic map preserving an ergodic SRB measure $\nu$. Let $\Lambda\subset X$, of positive $\nu$-measure, have \emph{hyperbolic product structure}, i.e.\ $\Lambda$ is the intersection of a family of stable manifolds with a family of unstable manifolds.

For $x\in X$ the first return time $\Rtn(x;T,\Lambda)=\min\{n\geq 1:T^n(x)\in\Lambda\}$ denotes the first iterate of $x$ to enter, or return to, the set $\Lambda$. Ergodicity ensures such a value exists almost everywhere. Of the successive returns of $x$, the first to satisfy an additional, technical condition on the length of local invariant manifolds (we omit the details, for which see the original papers) will be denoted by $\Rtn^*(x;T,\Lambda)$ and called the first \emph{good} return.

\begin{theo}[\cite{young2}]\label{thm:poly_decay}
If there exists $a>0$ such that
$$l\label{eqn:young}
\nu\{x\in X:\Rtn^*(x;T,\Lambda)>n\}=\mathcal{O}(n^{-a}),
$$l
then for $\alpha$-H\"{o}lder observables $f,g$ we have
$$
|C_n(f,g,T,\nu)|=\mathcal{O}(n^{-a}).
$$
\end{theo}

In practice constructing $\Lambda$ and establishing \eqref{eqn:young} can be prohibitively difficult. For the Arnold Cat Map, the procedure is described explicitly in \cite{chernov2000decay}, but this is a particularly straightforward construction, relying on the uniform hyperbolicity and linearity of the map.  A more tractable method is to first find a set $Y\subset X$ where hyperbolicity is `strong', so that we can choose $\Lambda\subset Y\subset X$, and so that the induced map $T_Y:Y\to Y$ defined by first returns satisfies
$$l\label{eqn:young2}
\nu\{x\in X:\Rtn^*(x;T_Y,\Lambda)>n\}=\mathcal{O}(\theta^n),
$$l
for some $\theta\in(0,1)$. This can be achieved, without needing to explicitly construct $\Lambda$, by establishing a few conditions first given by Chernov \cite{chernov} and later improved upon by Chernov and Zhang \cite{cz}. These conditions are reproduced at the end of the present section.

Finally \eqref{eqn:young} can be established from \eqref{eqn:young2} by a method essentially owing to Markarian \cite{markarian2}. The method, developed a little in \cite{cz}, introduces a redundant logarithmic factor to the decay rate, but this problem is resolved by a general scheme of Chernov and Zhang \cite{cz2}.

We now list the conditions given in \cite{cz} that collectively establish \eqref{eqn:young2}.

\subsection*{Smoothness}

$X$ is an open domain in a smooth ($C^{\infty}$) two-dimensional compact Riemannian manifold. The possibility of points at which $T$ is undefined, discontinuous and/or non-differentiable is admitted; in this case such points are contained within a closed set $D$ of zero Lebesgue measure. We refer to $D$ as the \emph{singularity set}. We denote by $D_m=\bigcup_{i=0}^{m-1}T^{-i}(D)$ the singularity set for $T^m$ and by $D_{-m}=\bigcup_{i=0}^{m-1}T^i(D)$ the singularity set for $T^{-m}$.

\subsection*{Hyperbolicity}

There are two families of cones $C^u(x)$ and $C^s(x)$ in the tangent space $T_xX$ for $x\in\overline{X}$. These families are \emph{continuous} on $\overline{X}$ and the angle between complementary cones is bounded away from zero. They are \emph{invariant} in the sense that $DT(C^u(x))\subset C^u(T(x))$ and $DT(C^s(x))\supset C^s(T(x))$ whenever $DT$ exists, and they are \emph{expanded} in the sense that
$$
\|DTv\|\geqs\lambda\|v\|\text{ for all }v\in C^u(x)\and \|DT^{-1}v\|\geqs\lambda\|v\|\text{ for all }v\in C^s(x),
$$
where $\lambda>1$ is a constant and $\|\cdot\|$ the Euclidean norm. For $m>0$, all tangent vectors to $D_m$ lie in stable ($C^s$) cones and all tangent vectors to $D_{-m}$ lie in unstable ($C^u$) cones.

If $\nu'$ is a $T$-invariant probability measure then $\nu'$-a.e.\ $x\in X$ has one positive and one negative Lyapunov exponent as well as one stable and one unstable manifold. We denote these $W^s(x)$ and $W^u(x)$ respectively.

\subsection*{SRB measure}

$T:X\to X$ preserves a mixing measure $\nu$ whose conditional distributions on unstable manifolds are absolutely continuous, i.e.\ $\nu$ is an SRB measure.

\subsection*{Distortion bounds}

Let $\lambda(x)$ denote the factor of expansion on unstable manifold $W^u$ at $x\in X$. If $x,y$ belong to the same unstable manifold $W^u$ and if $T^n$ is defined and smooth on $W^u$ then
$$
\log\prod_{i=0}^{n-1}\frac{\lambda(T^i(x))}{\lambda(T^i(y))}\leqs\xi(\d(T^n(x),T^n(y))),
$$
where $\d(\cdot,\cdot)$ denotes distance on $X$ and $\xi:\mathbb{R}^+\to\mathbb{R}^+$ is some function, independent of the choice of $W^u$, so that $\xi(t)\to 0$ as $t\to 0$.

\subsection*{Bounded curvature}

The curvature of unstable manifolds is uniformly bounded by a constant $B\geqs 0$.

\subsection*{Absolute continuity}

If $W_1,W_2$ are small, close unstable manifolds then the holonomy map $h:W_1\to W_2$, defined (where applicable) by sliding along stable manifolds, is absolutely continuous with respect to the Lebesgue measures (induced by the Euclidean metric) on $W_1$ and $W_2$. Moreover the Jacobian is bounded, i.e.\
$$
\frac{1}{C}\leqs\frac{\nu_{W_2}(h(W_1'))}{\nu_{W_1}(W_1')}\leqs C,
$$
for some $C>1$. Here $W_1'\subset W_1$ denotes those points at which $h$ is defined.

\subsection*{Structure of the singularity set}

We say that $W\subset X$ is an \emph{admissible curve in the unstable cone field}, or more concisely an \emph{unstable curve}, if all tangent vectors to $W$ are in unstable cones. For any admissible curve $W$ the set $W\cap D$ is at most countable and has at most $K$ accumulation points on $W$, $K$ being a constant. Moreover if $\{x_n\}_{n\in\mathbb{N}}\subset W\cap D$ is a sequence converging to an accumulation point $x_{\infty}$ then
$$
\d(x_n,x_{\infty})\leqs\text{const}\cdot n^{-d}
$$
for some constant $d>0$. 

\subsection*{One-step growth of unstable manifolds}

Let $W$ be a local unstable manifold, denote by $W_i$ the connected components of $W\back D$ and let $\lambda_i=\min\{\lambda(x):x\in W_i\}$, which is the minimal local expansion factor of $T$ on $W_i$. We have
$$
\liminf_{\delta\to 0}\sup_{W:|W|<\delta}\sum_i\lambda_i^{-1}<1,
$$
where $|W|$ denotes the length of unstable manifold $W$ and the supremum is taken over all unstable manifolds. The condition describes strong expansion along unstable manifolds. If $T$ is not sufficiently expansive it is enough that $T^m$ satisfies the condition for some $m\in\mathbb{N}$, with $D$, $W$ and $\lambda$ appropriately redefined.

\subsection*{}

This completes the list of conditions.

\section{The natural partition of the induced map}\label{Partition}

To prove Theorem~\ref{thm:main} we show that \eqref{eqn:young} is satisfied, with $(H,R,\mu)$ taking the place of $(T,X,\nu)$. To that end we take $Y=S$ and establish \eqref{eqn:young2} by verifying the conditions listed in Section~\ref{Background}. This occupies the present section and the following two. In Section~\ref{Proof} we establish \eqref{eqn:young} as described, however we do not need to appeal to \cite{cz2} in order to avoid redundant factors, rather we can employ an instructive result, Lemma~\ref{lem:isolation}, of the present section.

\begin{figure}
\centering
\begin{minipage}{0.31\linewidth}
\subfigure[Singularity set for $F_S$]{\includegraphics[angle=270,width=\linewidth]{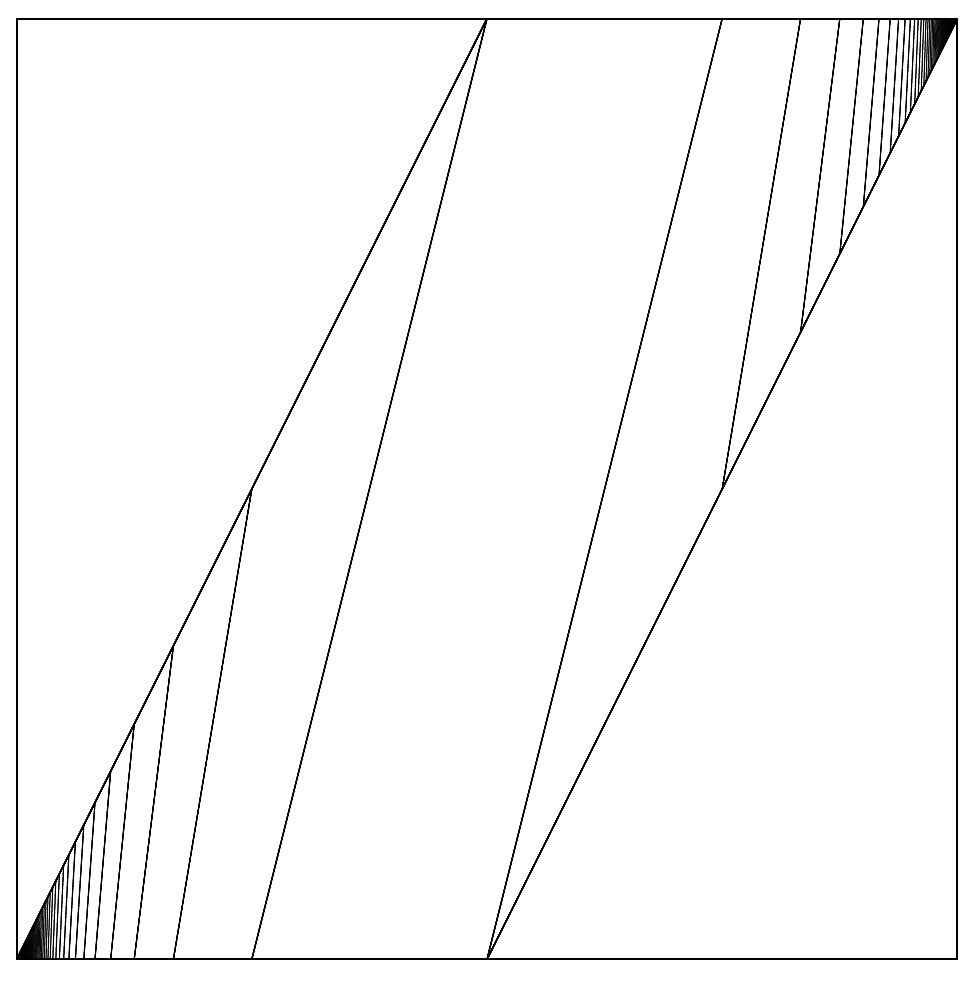}\label{fig:part_1}}\\
\subfigure[Singularity set for $G_S$]{\includegraphics[angle=270,width=\linewidth]{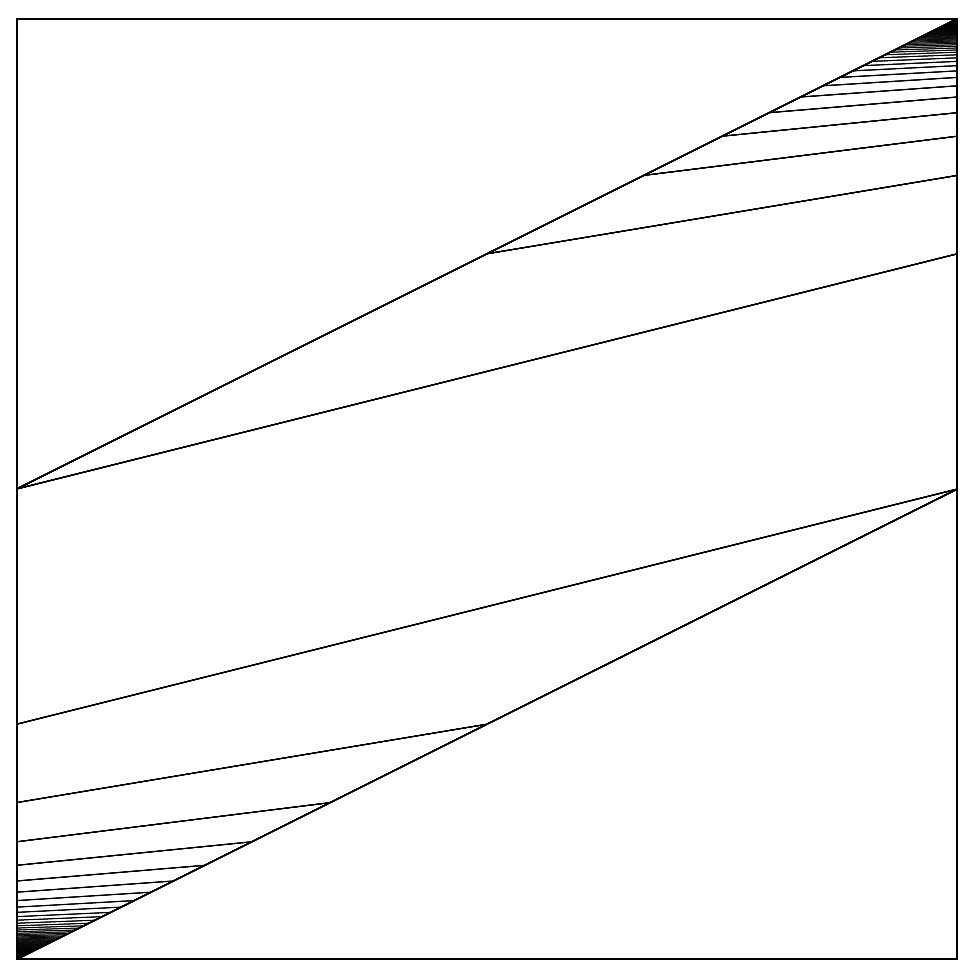}\label{fig:part_2}}
\end{minipage}\hfill
\begin{minipage}{0.68\linewidth}
\subfigure[Singularity set for $H_S$, denoted $\sigma$]{\includegraphics[width=\linewidth]{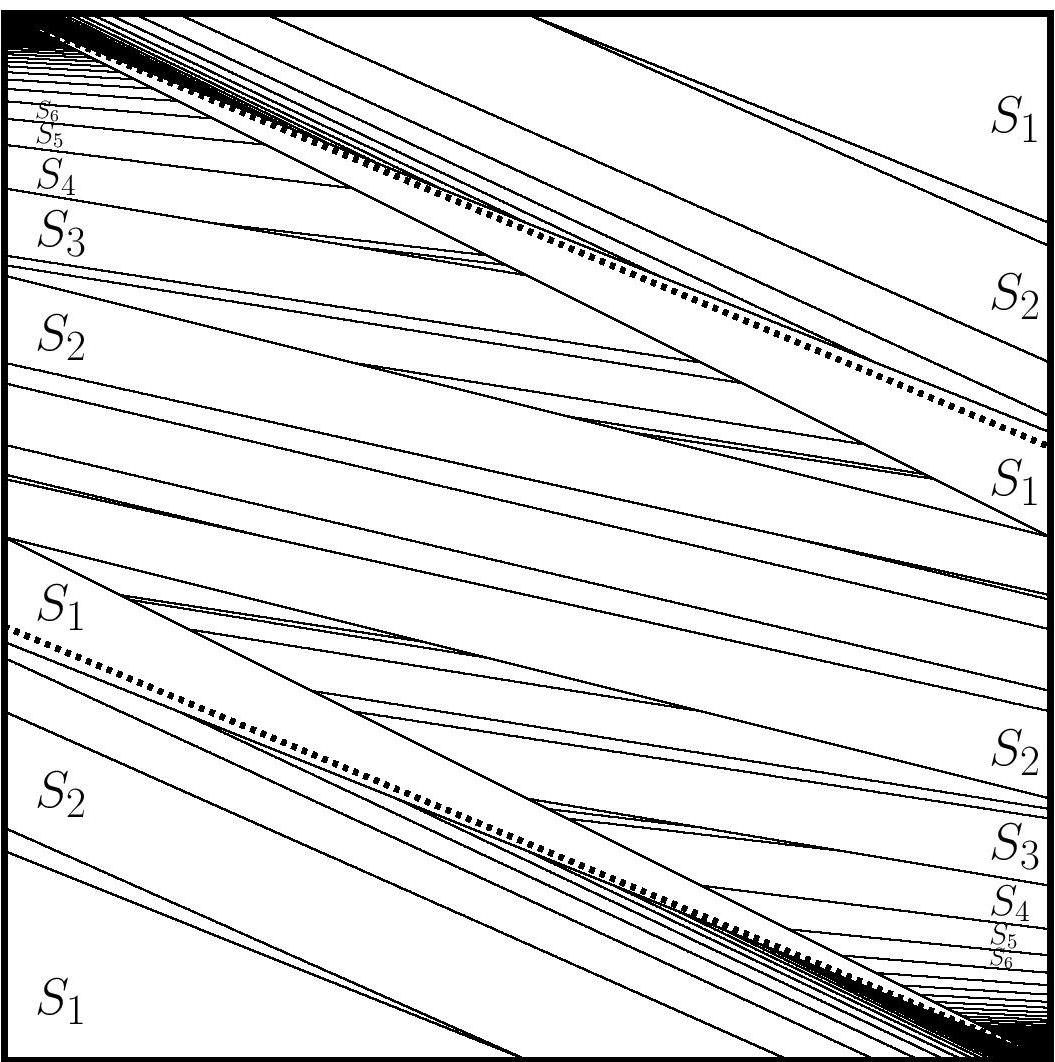}\label{fig:part_3}}
\end{minipage}
\caption{Singularity sets for the return maps (a) $F_S$, (b) $G_S$ and (c) $H_S$. Details of the constuction of singularity sets are given in appendix~\ref{app:A}.  The dotted lines in figure (c) show the unstable manifolds of points $p$ and $q$, and the components of the sets $S_1$ and $S_2$ are labelled. The sets $S_n$ for large $n$ can be seen accumulating in the top-left and bottom-right corners of $S$ (points $q$ and $p$ respectively), with some components of $S_3, S_4, S_5, S_6$ labelled. Note that these regions in the singularity set for $H_S$ in which the $S_n$ accumulate on $p$ and $q$ correspond to the relevant regions in the singularity sets for $F_S$ and $G_S$, and so their structure can be easily found explicitly.}\label{fig:part}
\end{figure}

We begin the analysis of $H_S$ by considering the natural partition it induces on $S$. Let $F_S:S\to S$, the return map with respect to the twist $F$, be given by
$$
F_S(z)=F^n(z),\quad\text{where }n=\Rtn(z;F,S).
$$
Similarly define $G_S:S\to S$. It is easily checked that
$$
H_S=G_S\circ F_S.
$$
Those points $z\in S$ for which $\Rtn(z;F,S)$ is `large' are confined to neighbourhoods of the corners $p=(1,0), q=(0,1) \in S$. The same is true for $\Rtn(z;G,S)$ and $\Rtn(z;H,S)$. Singularity sets for $F_S$, $G_S$ and $H_S$ are shown in Figures~\ref{fig:part_1}, \ref{fig:part_2} and~\ref{fig:part_3} respectively. Dotted lines represent local stable manifolds of $p$ and $q$, which lie within the set $S_1$ of points which return to $S$ under a single iterate of $H$. The components of the set $S_2$ are also labelled. The structure of these singularity sets are given in appendix~\ref{app:A}, although the majority of our arguments do not require the precise geometrical details.

We denote by $\sigma$ the singularity set for $H_S$ and by $\sigma^n$ the singularity set for $H_S^n$, $n\in\mathbb{Z}$. The set $\sigma$ partitions $S$ into countably many open sets on which $H_S$ is a linear map characterised by a hyperbolic matrix
\begin{equation}\label{eqn:DH_S}
DH_S=DG_S\cdot DF_S=\left(\begin{array}{cc} 1 & 0 \\ 2 & 1 \end{array}\right)^k\left(\begin{array}{cc} 1 & 2 \\ 0 & 1 \end{array}\right)^j=\left(\begin{array}{cc} 1 & 2j \\ 2k & 4jk+1 \end{array}\right).
\end{equation}
Here $j=\Rtn(z;F,S)$ and $k=\Rtn(F_S(z);G,S)$. We remark that $\Rtn(z;H,S)=j+k-1$.

Let 
$$
S_n=\left\{z\in S:\Rtn(z;H,S)=n\right\}.
$$
Clearly $S\back\bigcup_{n=1}^{\infty}S_n$ has zero $\mu$-measure. Let $\mu_S$ denote the restriction of $\mu$ to $S$.

\begin{prop}\label{prop:ergodic}
$\mu_S$ is an ergodic, invariant measure for $H_S$.
\end{prop}

The result is entirely standard and we omit a proof. What is not immediately clear is that $H_S$ is in fact Bernoulli; we prove this in Section~\ref{Bernoulli}.

For $z\in S$ let $(u,v)=(\d x,\d y)$ give coordinates in the tangent space $T_zS=\mathbb{R}^2$. In $T_zS$ we define a pair of cones
$$
C^+(z)=\{(u,v):u=0\text{ or }v/u\geqs 1\},\quad C^-(z)=\{(u,v):v=0\text{ or }-u/v\geqs 1\},
$$
called the \emph{unstable} and \emph{stable} cones at $z$ respectively. These are illustrated in Figure \ref{fig:cones}. The cones are independent of the underlying point $z$ and the angle between them is bounded away from zero. Define the unstable and stable cone-fields:
$$
C^{\pm}=\bigcup_{z\in S}C^{\pm}(z).
$$
Our next result says that $C^{\pm}$ is invariant under and uniformly expanded by the derivative $DH_S^{\pm 1}$. Moreover $C^{\pm}$ contains all tangent vectors to the singularity set for $H_S^{\mp 1}$. Let $\|\cdot\|:\mathbb{R}^2\to[0,\infty)$ denote the standard Euclidean norm.
\begin{figure}
\centering
\includegraphics[width=0.75\textwidth]{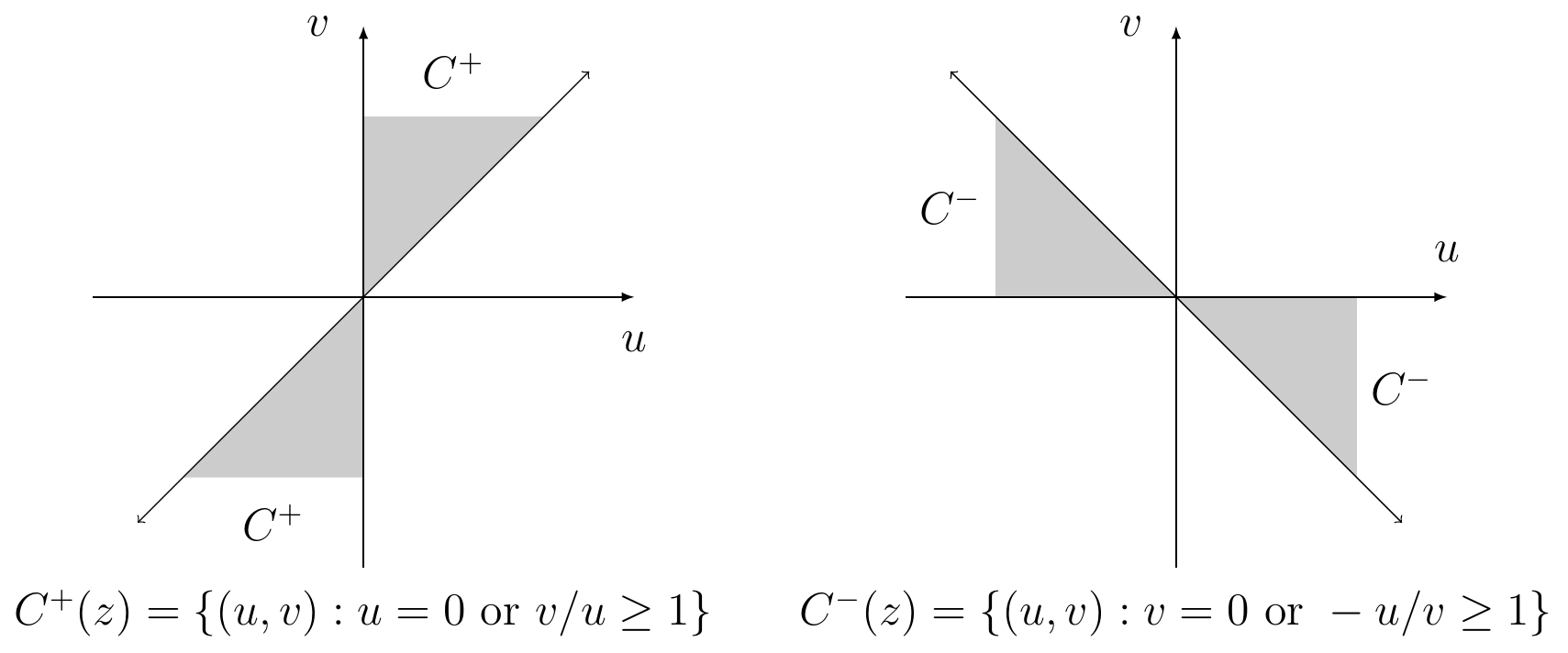}
\caption{The unstable and stable cones $C^+(z)$ and $C^-(z)$.}\label{fig:cones}
\end{figure}

\begin{lemma}[Hyperbolicity]\label{lem:cones}
If $w\in C^{\pm}$ is a tangent vector to $z\in S\back\sigma^{\pm 1}$ then
$$
(DH_S^{\pm 1})_zw\subset C^{\pm}(H_S^{\pm 1}(z))\and \|(DH_S^{\pm 1})_zw\|> \sqrt{5}\|w\|.
$$
Conversely, if $w$ is a tangent vector to $z \in \sigma^{\pm 1}$ and is tangent to a singularity line through that point, then $w \in C^{\mp}$.
\end{lemma}

\begin{proof}
For the first statement, we deal only with the case $w \in C^+$, the other being entirely similar. Fix $z\in S\backslash \sigma$, let $w=(u,v)^T\in C^+(z)$ be a tangent vector at $z$ and le $w'=(u',v')^T=DH_S(u,v)^T$ be a corresponding tangent vector at $H_S (z)$. Using (\ref{eqn:DH_S}) we have 
$$
\frac{v'}{u'}>\frac{2ku+4jkv}{u+2jv}=2k>1,
$$
i.e.\ $w'\in C^+(H_S(z))$. Moreover
$$
u'^2+v'^2>u^2\left(1+4k^2\right)+v^2\left(4j^2+\left(4jk+1\right)^2\right)>5\left(u^2+v^2\right),
$$
showing that $\|w'\|>\sqrt{5}\|w\|$.

For the second statement, we deal only with the case $z \in \sigma^{-1}$, the other being entirely similar.  The boundary of $S$ consists of horizontal and vertical lines, so if $(u,v)^T$ is tangent to the boundary then either $u=0$ or $v=0$. The singularity set $\sigma^{-1}$ consists of the $H_S$-images of the boundary, therefore its tangents consist of the $DH_S$-images of those tangents vectors just described. (\ref{eqn:DH_S}) gives
$$
DH_S(0,v)^T=(2jv,(4jk+1)v)^T\and DH_S(u,0)^T=(u,2ku)^T,
$$
and it is easily shown that these vectors are in $C^+$ as required.
\end{proof}
\medbreak

Our next result concerns the `itinerary' of $z\in S$ with respect to $H_S$. We show that if $n$ is large and $z\in S_n$ then some number, depending only on $n$, of the immediate pre-images and images of $z$ must be in $S_1$ (recall Figure \ref{fig:part_3}). In effect long returns are \emph{isolated}. This feature of the dynamics turns out to be crucial in establishing the polynomial decay rate.

\begin{lemma}[Isolation of large return times]\label{lem:isolation}
There are constants $K,k>0$ so that if $z\in S_n$ and $n>Ke^{kN}$ then $H_S^i(z)\in S_1$ for each $1\leqs |i|\leqs N$.
\end{lemma}

\begin{proof}
Roughly speaking, if $n$ is large then $S_n$ is close to $p$ or $q$ and $H_S(S_n)$ is close to $0$ or $s$, respectively. The (exponential) rate at which successive images move away is bounded and so some number of iterates remain in $S_1$. The behaviour of $H_S^{-1}$ is similar.

We deal rigorously with one case. Let $S_n'\subset S_n$ be the connected component close to $p$ and having $x=1$ as a boundary. Let $S_1'\subset S_1$ be the connected component adjacent to $0$. By considering the map $H$ and the partition of $S$ in Figure~\ref{fig:part_3} we observe that
\begin{enumerate}
\item[(i)] there exists $\lambda>1$ such that if $z\in S_1'$ then $\d(0,H_S(z))\leqs\lambda\d(0,z)$,
\item[(ii)] there exists $c>0$ such that if $z\in S_n'$ then $\d(0,H_S(z))\leqs c/n$,
\item[(iii)] there exists $C>0$ such that if $\d(z,0)<C$ then $z\in S_1'$.
\end{enumerate}

Now suppose that $z\in S_n'$ and $n>\frac{c}{C}\lambda^N$. It follows from, in turn, (i), (ii) and the assumption on $n$ that
$$
\d\left(0,H_S^{i+1}(z)\right)\leqs\lambda^i\d(0,H_S(z))\leqs\lambda^i\frac{c}{n}\leqs C\lambda^{i-N},
$$
and finally from (iii) that $H_S^{i+1}(z)\in S_1'$ for each $0\leqs i\leqs N$.
\end{proof}
\medbreak

\section{The Bernoulli property for the induced map}\label{Bernoulli}

In this section we prove the following:
\begin{theo}\label{thm:Bernoulli}
$H_S$ is Bernoulli.
\end{theo}

Fundamental results concerning the ergodic properties of non-uniformly hyperbolic systems were established by Pesin \cite{pes} and extended to a class of \emph{smooth maps with singularities} by Katok \emph{et.\ al.}\ \cite{ks}. We describe some results of the latter, restricting ourselves to the two-dimensional case.

Let $T$ be a map defined on an open subset $X$ of a compact Riemannian manifold and preserving a measure $\nu$. $T$ is smooth except possibly for a set of singularities (points of discontinuity or non-differentiability) contained within a union $D$ of smooth, compact submanifolds of positive codimensions. The `heaviness' of $D$ is restricted thus: there are positive constants $a,C_1$, and for any $\eps>0$
$$l\label{eqn:KS1}
\nu\left(B_{\eps}(D)\right)\leqs C_1\eps^a.
$$l
$B_{\eps}(D)$ is the $\eps$-neighbourhood of $D$ using the Riemannian metric. This is commonly refered to as the condition (KS1). A further condition (KS2) requires an upper bound on the growth of second derivatives in the vicinity of $D$, however it requires a little notation to give a precise formulation and will be trivially satisfied by piecewise-linear systems such as ours, so we do not state it. The following condition of Oseledec \cite{osel} is also required:
$$l\label{eqn:OS}
\int_X\log^+\|DT\|\d\nu<\infty,
$$l
where $\log^+(x):=\max\{\log(x),0\}$ for $x\in X$. We call it condition (OS).

If $T:X\to X$ as above satisfies (KS1), (KS2) and (OS), then Lyapunov exponents
$$
\chi_{\pm}(x,t)=\lim_{n\to\pm\infty}\frac{1}{n}\log\|DT^n(x)t\|
$$
exist for $\nu$-a.e.\ $x\in X$ and each tangent vector $t$ at $x$. If there is one positive and one negative Lyapunov exponent at $x$ then $x$ has a local unstable manifold $\gamma^u(x)$  and a local stable manifold $\gamma^s(x)$, and these are absolutely continuous. $X$ has an \emph{ergodic partition}, meaning that $X=\bigcup_{i=1}^{\infty}X_i$ where $T|_{X_i}$ is ergodic for each $i$ and, moreover, each $X_i=\bigcup_{j=1}^{n(i)}X_{i,j}$ where $T^{n(i)}|_{X_{i,j}}$ is Bernoulli for each $j$.

All of these conclusions hold equally for $T^m$, $m\in\mathbb{N}$.

\begin{lemma}\label{lem:KS}
$H_S:S\to S$ satisfies (\ref{eqn:KS1}), i.e.\ there are $a>0$, $C_1>0$, and for any $\eps>0$
$$
\mu_S\left(B_{\eps}(\sigma)\right)\leqs C_1\eps^a.
$$
\end{lemma}

\begin{proof}
Singularity line-segments accumulate in the four `groups' shown in Figure~\ref{fig:part_3}. We consider one such group, shown in Figure~\ref{fig:eps_balls_a}, and establish an appropriate bound. The lemma follows easily. Excluding $L$, index these line-segments $\sigma_n$, $n\in\mathbb{N}$, in order of decreasing length. Notice that\footnote{Here we use $\sim$ to indicate the common asymptotic notation given by $f \sim g $ if $f/g  \to 1$ for functions $f$ and $g$.} $\text{length}(\sigma_n)\sim 1/n$; in particular, the total length is unbounded.

\begin{figure}
\subfigure[]{\label{fig:eps_balls_a}\includegraphics[width=0.49\linewidth]{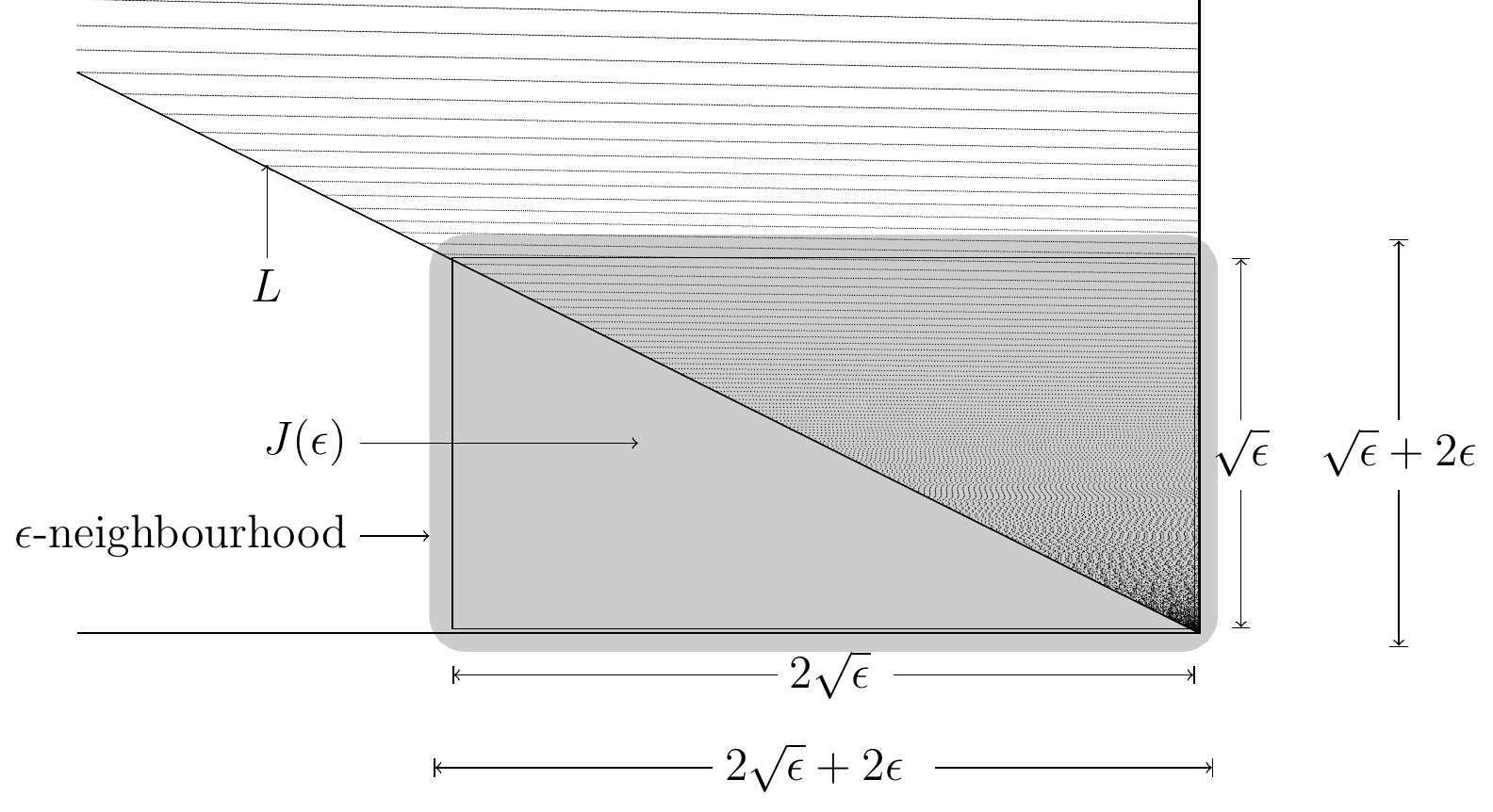}}\hfill
\subfigure[]{\label{fig:eps_balls_b}\includegraphics[width=0.49\linewidth]{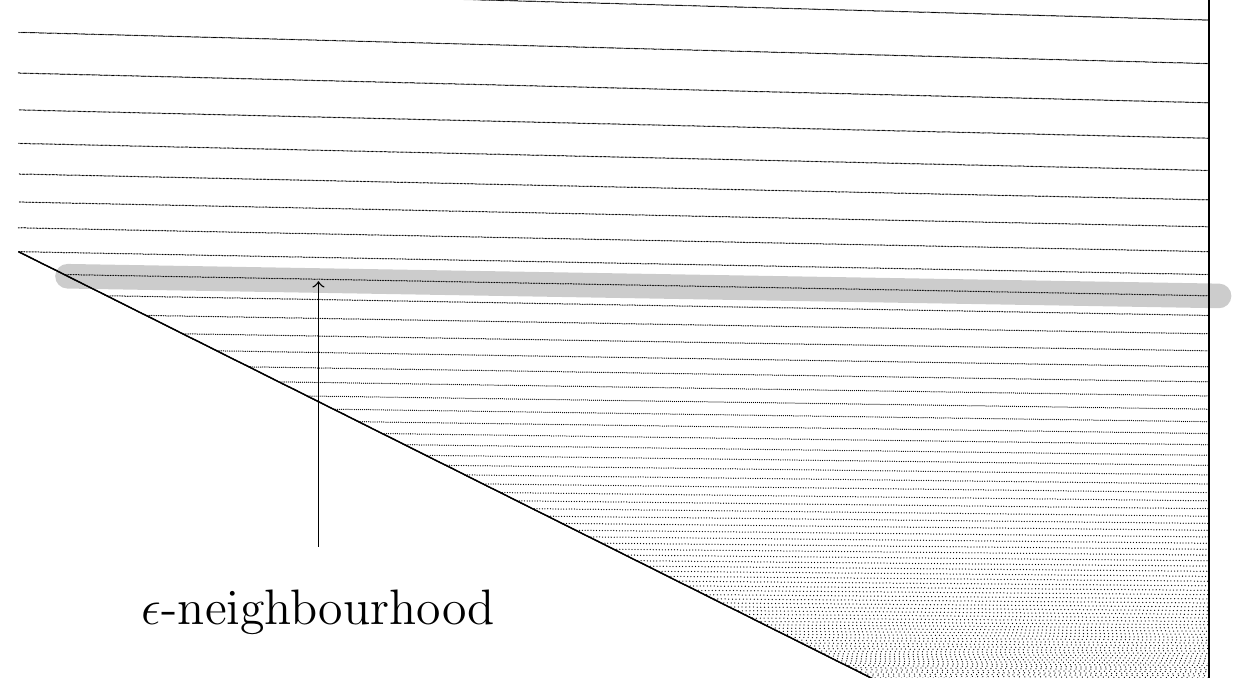}}\hfill
\caption{Illustration of the $\eps$-neighbourhood for the singularity set $\sigma^F$.}\label{fig:eps_balls}
\end{figure}

Fix a small $\eps>0$. We construct the $\eps$-neighbourhood in two stages. Let $J(\eps)$ be a rectangle with sides parallel to coordinate directions; of height $\sqrt{\eps}$ and width $2\sqrt{\eps}$; and with lower-right vertex at $(1,0)$, coinciding with that of $S$. See Figure~\ref{fig:eps_balls_a}. Now let $N$ be the smallest integer to exceed $1/2\sqrt{\eps}$, then $J(\eps)$ contains $\sigma_n$ for every $n\geqs N$ and
$$
\mu_S\left(B_{\eps}(J(\eps))\right)=\left(\sqrt{\eps}+2\eps\right)\left(2\sqrt{\eps}+2\eps\right)=2\eps+6\eps^{3/2}+4\eps^2<12\eps.
$$

There remain $N-1$ line-segments to consider. An $\eps$-neighbourhood of one such segment is shown in Figure~\ref{fig:eps_balls_b}. The total measure is at most
$$
\sum_{n=1}^{N-1}\mu_S\left(B_{\eps}\left(\sigma_n\right)\right)
\sim 2\eps\sum_{n=1}^{\lfloor 1/2\sqrt{\eps}\rfloor}\left(\frac{1}{n}+2\eps\right)
\sim 2\eps^{3/2}+\eps\ln\frac{1}{4\eps}.
$$
The asymptotic notation describes the limit $\eps\to 0$ (equivalently $N\to\infty$) and we have used the fact that $\sum_{n=1}^N1/n\sim\ln N$. Finally we observe that as $x\to\infty$ the polynomial $x^b$, $b>0$ grows more quickly than $\ln x$. Substituting $x=1/\eps$, this says that as $\eps\to 0$ the polynomial $\eps^a$, $a<1$ grows more quickly than $\eps\ln\frac{1}{\eps}$. Thus the lemma holds for any $0<a<1$ and an appropriate $C_1$.

\end{proof}
\medbreak

\begin{lemma}\label{lem:OS}
$H_S:S\to S$ satisfies (\ref{eqn:OS}), i.e.
$$
\int_S\log^+\|DH_S\|\d\mu_S<\infty.
$$
\end{lemma}

\begin{proof}
$\|DH_S(z)\|$ is given by the largest eigenvalue of a hyperbolic matrix as in (\ref{eqn:DH_S}). If $z\in S_n$ and $n$ is large then either $H_S(z)=G\circ F^n(z)$ or $H_S(z)=G^n\circ F(z)$ (this is an easy consequence of the dynamical features described in Lemma~\ref{lem:isolation}). In either case the eigenvalue in question is $1+2n+\sqrt{4n(n+1)}\sim 4n$. Thus
$$
\int_S\log^+\|DH_S(z)\|\d\mu_S\leqs\text{const}\sum_{n=1}^{\infty}\log^+(4n)\mu(S_n).
$$
$S_n$ (see Figure~\ref{fig:eps_balls}) is approximately rectangular, having width $\sim 1/n$ and height $\sim 1/(n+1)-1/n=\mathcal{O}\left(1/n^2\right)$ so that $\mu_S(S_n)=\mathcal{O}\left(1/n^3\right)$. Thus the sum converges.
\end{proof}
\medbreak

For $\mu_S$-a.e.\ $z\in S$ and for each non-zero $w\in T_zS$, Lemmas~\ref{lem:KS} and~\ref{lem:OS} establish the existence of Lyapunov exponents
$$
\chi_{\pm}(z,w)=\lim_{n\to\pm\infty}\frac{1}{n}\log\|DH_S^n(z)w\|
$$
for $H_S$. Their existence for $H_S^m$, $m\in\mathbb{N}$, is an easy consequence.

\begin{lemma}\label{lem:LE}
Lyapunov exponents for $H_S:S\to S$ are non-zero.
\end{lemma}

\begin{proof}
It is a standard result that for a $\mu_S$-typical $z\in S$ there can be at most two distinct Lyapunov exponents. For such a $z$ let $w\in C^u(z)\subset T_zS$. Lemma~\ref{lem:cones} shows that $DH_S^n(z)w\in C^u(z)$ and that $\|DH_S^n(z)w\|> 5^{n/2}\|w\|$, for each $n\in\mathbb{N}$. Thus
$$
\chi_+(z,w)>\lim_{n\to\infty}\frac{1}{n}\log 5^{n/2}\|w\|=\frac{1}{2}\log 5>0,
$$
showing that $z$ has a positive Lyapunov exponent. Similarly, taking $w'\in C^s(z)$ leads to the conclusion that $\chi_-(z,w')<0$ and thus that $z$ also has a negative Lyapunov exponent.
\end{proof}
\medbreak

We now consider global properties of stable and unstable manifolds to prove the main theorem of this section. For clarity, we define $\gamma^u(z)$ and $\gamma^s(z)$ to be local unstable and stable manifolds, respectively, of $z$; by definition these are connected. To prove theorem~\ref{thm:Bernoulli} we will confirm the intersection of forward iterates of $\gamma^u(z)$ with backward iterates of $\gamma^s(z)$. In section~\ref{OneStep} we will study global unstable and stable manifolds, defined by $W^u(z) = \bigcup_{n\ge 0} H^n_S \gamma^u(z)$ and $W^s(z) = \bigcup_{n\ge 0} H^{-n}_S \gamma^s(z)$ respectively. These are only piecewise connected as they are cut under iteration as described in the following proof.

We note that the corresponding manifolds for manifolds for $H$ can be computed explicitly, as described in~\cite{woj}, as straight lines with gradient given by a continued fraction whose entries are given by successive return times to $S$. A similar approach could be taken for $H_S$, but we only require to observe that gradients of such straight lines are constrained to lie within the cones $C^+(z)$ and $C^-(z)$. 

\begin{proof}[Proof of Theorem~\ref{thm:Bernoulli}]

It is enough (\cite{ks} or \cite{p1}) that for $\mu$-a.e.\ $z,z'\in S$ and all $m,n\in\mathbb{N}$ large enough
$$l\label{eqn:RMIP}
H_S^n\gamma^u(z)\cap H_S^{-m}\gamma^s(z')\neq\emptyset .
$$l
This property is clearly satisfied by $H$. As discussed in \cite{p1,woj}, images of local unstable and stable manifolds, under iteration of $H$, diverge exponentially in length, remain connected, and are mutually transversal. As soon as they span $S$ in the vertical and horizontal directions respectively, the analogue of \eqref{eqn:RMIP} is guaranteed.

The difference in the argument for $H_S$ is that although images of local invariant manifolds diverge exponentially in length and are mutually transversal, they do not remain connected. Indeed when, under iteration by $H_S$, a segment of unstable manifold falls over more than one component of the singularity set $\sigma$, the image is cut into possibly countably many disconnected pieces, as in figure~\ref{fig:intersection}. We therefore proceed as follows: first, we confirm that exponentially lengthening segments remain, even in the presence of cutting; second, we discuss images of connected pieces who have grown to the size of $S_i$, for small $i$; finally, we consider $S_i$ for large $i$. 

\begin{figure}
\subfigure[Initial segment $\tilde{\gamma}$ (shown zoomed).]{\label{fig:intersection1}\includegraphics[width=0.4\linewidth,angle=270]{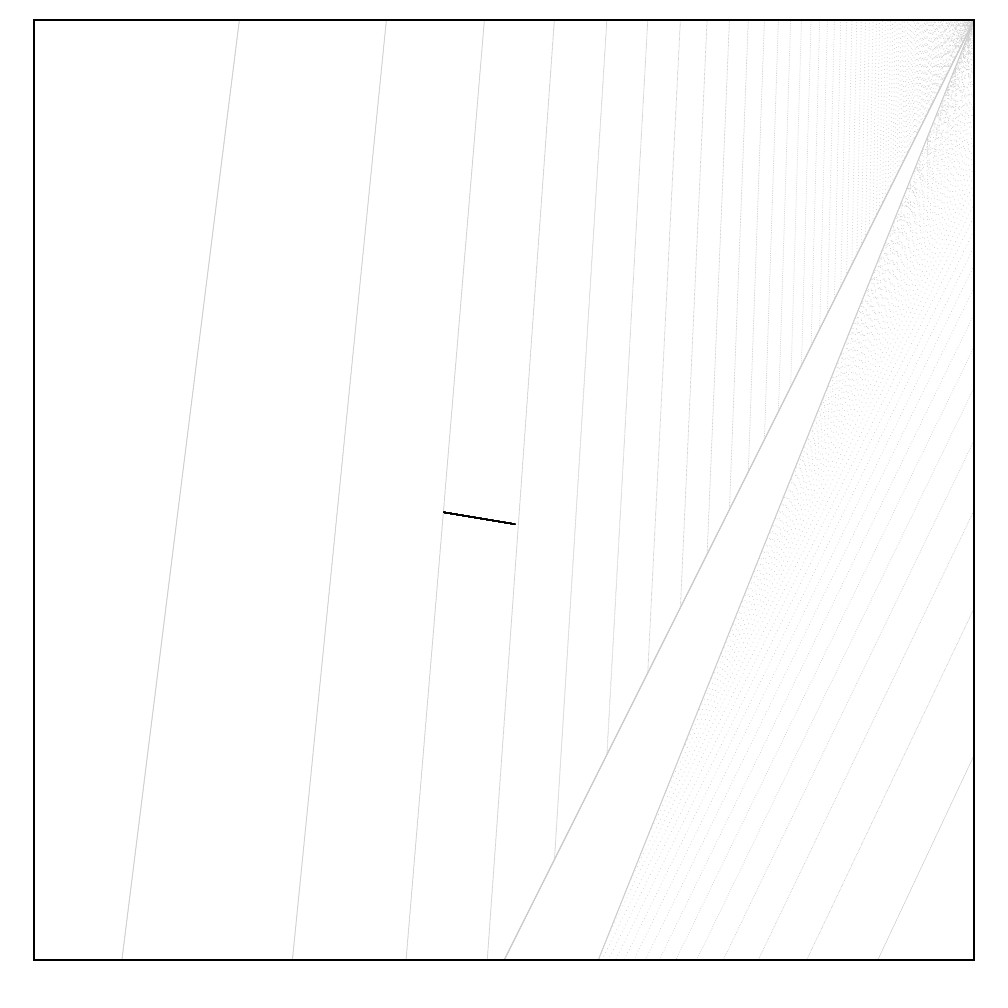}}\hfill
\subfigure[$H_S(\tilde{\gamma})$.]{\label{fig:intersection2}\includegraphics[width=0.4\linewidth,angle=270]{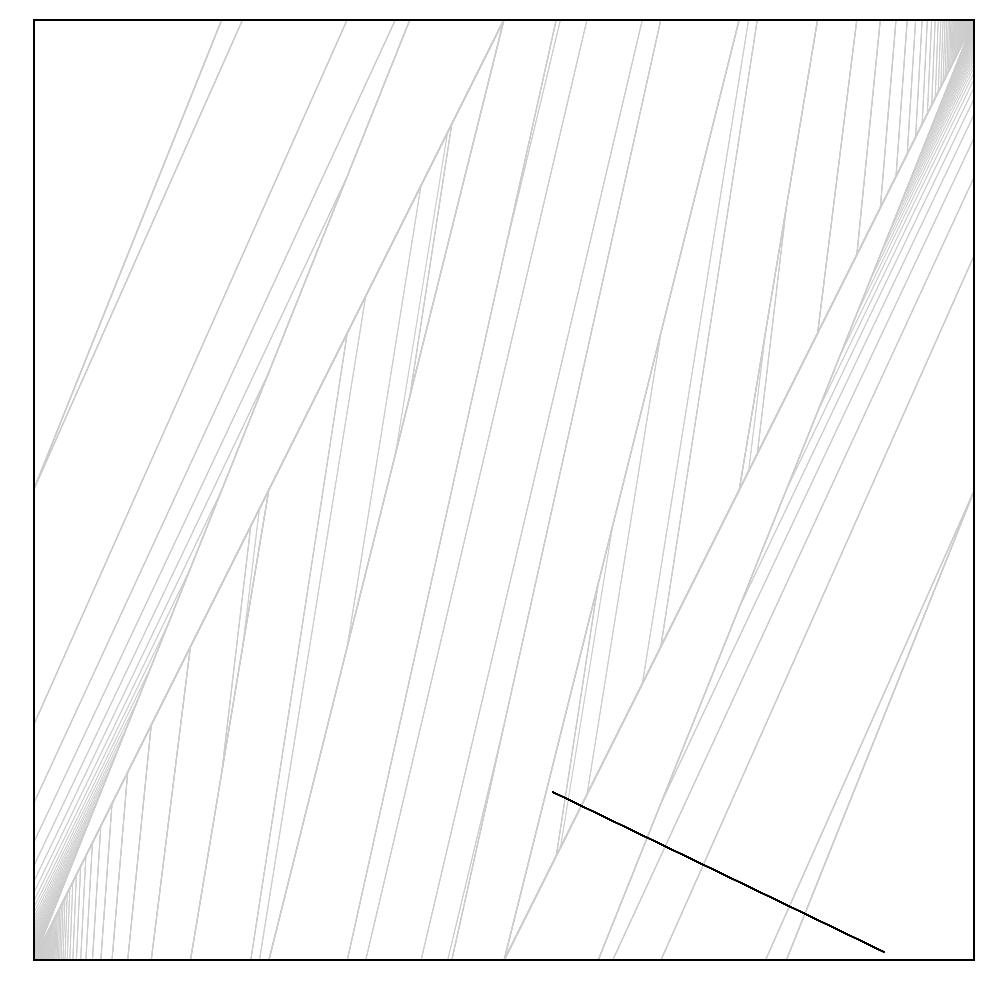}}\\
\subfigure[$H_S^2(\tilde{\gamma})$.]{\label{fig:intersection3}\includegraphics[width=0.4\linewidth,angle=270]{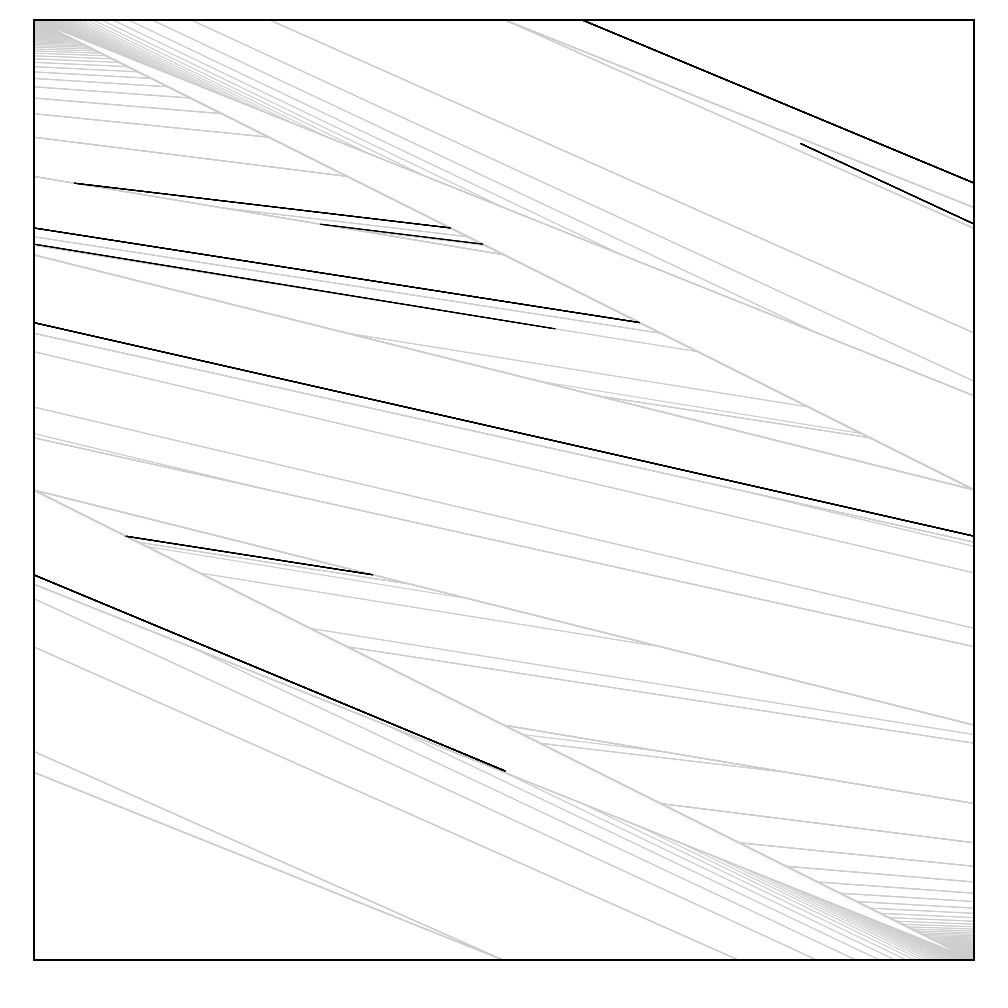}}\hfill
\subfigure[Stable manifold of $c=(1/2,1/2)$ intersecting $H_S^i (\tilde{\gamma})$.]{\label{fig:intersection4}\includegraphics[width=0.4\linewidth,angle=270]{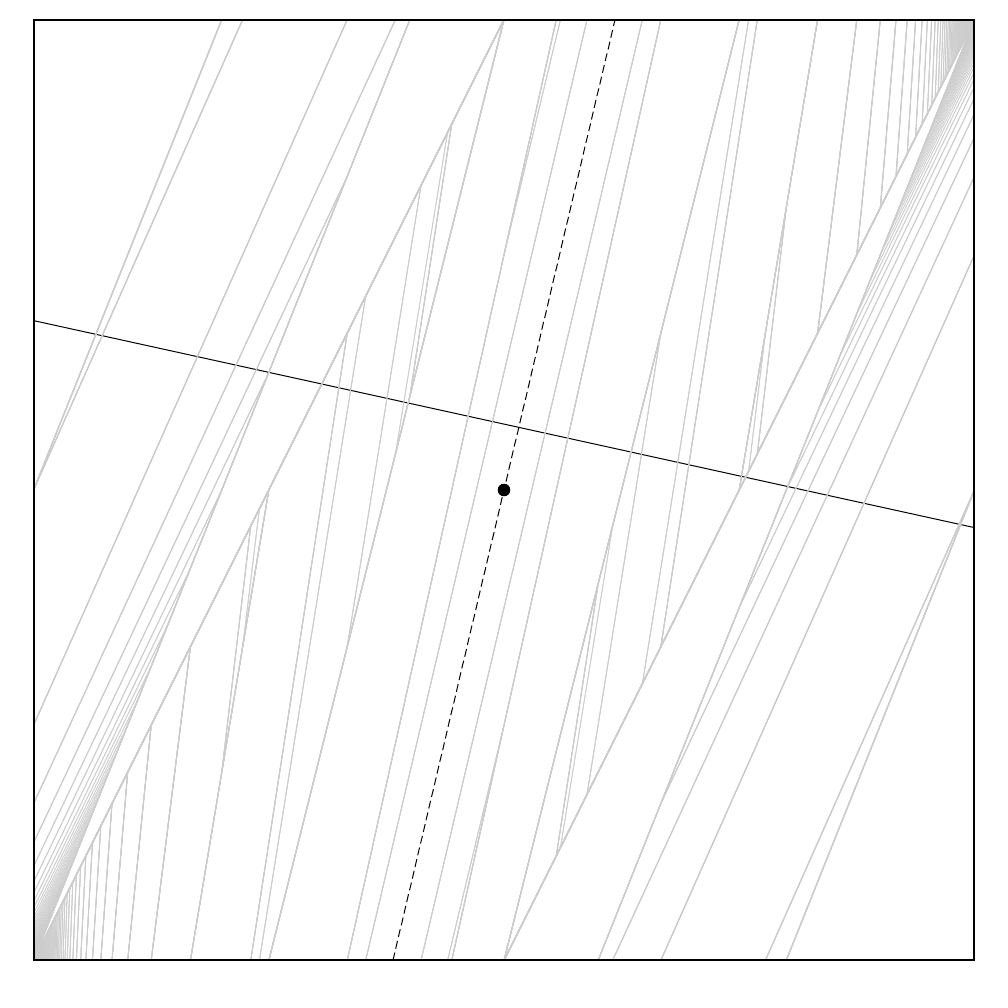}}\hfill
\caption{Images of a sample initial segment of unstable manifold $\tilde{\gamma}$.  \ref{fig:intersection1} shows the initial segment as a black line with $\sigma$ shown in grey, plotted in $[0.9,1]\times [0,0.1]$ for clarity. \ref{fig:intersection2} shows the image of $\tilde{\gamma}$ under $H_S$, with the following iterate shown in \ref{fig:intersection3}. In \ref{fig:intersection4} we show the hyperbolic fixed point (for $H_S$) $c$ with its stable manifold intersecting a connected component of $H_S^2(\tilde(\gamma))$. }\label{fig:intersection}
\end{figure}

Observe that the largest eigenvalue of $DH$ is $\sqrt{5} >2$ and write $l_v(\cdot)$ and $l_h(\cdot)$ for the vertical and horizontal lengths of a line segment respectively. Then, letting $\hat{\gamma}$ be a connected segment of $H_S^i(\gamma^u(z))$, we have that $l_v(H_S(\hat{\gamma}))>2l_v(\hat{\gamma})$. Hence if $H_S(\hat{\gamma})$ lies in exactly two elements $S_j$, then it contains a connected segment $\tilde{\gamma}$ such that $l_v(\tilde{\gamma}) > \delta l_v (\hat{\gamma})$, for some $\delta>1$.  Thus $H_S^n (\gamma(z))$ always contains, for each $n$, an exponentially growing connected segment, at least until a segment is cut into three or more pieces. If this occurs then a connected segment $\tilde{\gamma}$ has spanned one of the elements of $H_S$ from side to side.

Consider first the region outside $J(\epsilon)$, as shown in figure \ref{fig:eps_balls_a}, and its counterpart near $q$. This consists of a finite number of regions $S_i$. It is tedious but not difficult, due to the finiteness of the problem, to show that such an $S_i$-spanning segment $\tilde{\gamma}$ for each $i$ has the property that a connected segment of $H_S^k (\tilde{\gamma})$ stretches from $y=0$ to $y=1$ for some finite $k$. This can be verified by considering the images of the endpoints of $\gamma$ (these lie on lines of $\sigma$ described in appendix~\ref{app:A}). We show an illustrative example for a typical small initial $\tilde{\gamma}$ in figure \ref{fig:intersection}. We begin in figure~\ref{fig:intersection1} with a segment assumed to have grown to span some $S_i$. Since $\tilde{\gamma}$ lies entirely within an element $S_i$ of $\sigma$, $H_S(\tilde{\gamma})$ is again connected, and has grown by a factor $\lambda^i$. Careful counting will reveal that in this case $H_S(\tilde(\gamma))$ lies over nine elements of $\sigma$. \ref{fig:intersection3} shows the following iterate, which consequently contains nine connected components, and which contains a connected segment joining $y=0$ to $y=1$. In this figure that grey backdrop is the set $\sigma^{-1}$, illustrating how connected components of unstable manifolds are contained within the singularity set for $H_S^{-1}$. Finally in \ref{fig:intersection4} we show the hyperbolic fixed point (for $H_S$) $c$ with its stable manifold intersecting a connected component of $H_S^2(\tilde(\gamma))$, and indeed a connected component of all future iterates $H_S^i(\tilde(\gamma))$, $i\ge 2$.

According to~\cite{chernov}, the obstacle in demonstrating the Bernoulli property is usually the problem that a segment of unstable manifold may be cut into countably many pieces. In this dynamical system such a phenomenon occurs in $J(\epsilon)$, or its counterpart near $q$. Consider now a segment $\gamma$ which connects the top and bottom edges of $S_n$, for some sufficiently large $n$. It has height $l_v(\gamma) \ge \frac{1}{4(n-1)} - \frac{1}{4n} = \frac{1}{4n(n-1)}$. Under a single iterate of $H_S$, by definition, $\gamma$ undergoes $n$ iterates of $F$, producing $\gamma'$ which satisfies $l_v(\gamma') =l_v(\gamma) \ge \frac{1}{4n(n-1)}$ and $l_h(\gamma) \ge \frac{1}{2(n-1)}$. At this point we have a segment $\gamma'$ of length $ \sim \frac{1}{n}$ which now enjoys exponential growth until leaving the component of $S_1$ adjacent to $0$, whereupon it has grown sufficiently that its next iterate produces an $S_i$-spanning set, and the procedure above applies. Considering an identical argument for $H_S^{-1}$ this is sufficient to show that there exists $m,n\in\mathbb{N}$ large enough that
$$l\label{eqn:MIP}
H_S^n\gamma^u(z)\cap H_S^{-m}\gamma^s(z')\neq\emptyset .
$$l

To demonstrate (\ref{eqn:RMIP}), i.e., that the the equation above holds for {\em all} sufficiently large $m$ and $n$, consider $c=\left(1/2,1/2\right)$. It is a hyperbolic fixed point for $H_S$ (it is a period three point for $H$). From (\ref{eqn:MIP}), for all $z,z'\in S$ we have 
$$l\label{eqn:intersect}
H_S^n\gamma^u(z)\cap\gamma^s(c)\neq\emptyset\and H_S^{-m}\gamma^s(z')\cap\gamma^u(c)\neq\emptyset.
$$l
The intersections are transversal (see figure \ref{fig:intersection4}), so the inclination lemma says that successive $H_S$-images of $\gamma^u(z)$ accumulate on $\gamma^u(c)$, and successive $H_S^{-1}$-images of $\gamma^s(z)$ accumulate on $\gamma^s(c)$. Thus (\ref{eqn:RMIP}) follows from (\ref{eqn:intersect}).

\end{proof}

\medbreak

\section{Local expansion factors}\label{OneStep}

In this section we consider the \emph{one-step growth} condition introduced by Chernov and Zhang \cite{cz} and related in Section~\ref{Background}. $H_S$, it transpires, is not itself expansive enough and we are led to consider $H_S^2$ instead.

Let $W$ be a global unstable manifold (we suppress the superscript $u$ here for ease) for $H_S^2$, let $W_i$ be the connected components of $W\back\sigma^2$ and let $\lambda_i=\min\{\lambda(z):z\in W_i\}$. $\lambda_i$ is the minimal local expansion factor of $H_S^2$ on $W_i$. The piece-wise linearity of $H_S^2$ ensures that $\lambda$ is in fact constant on $W_i$. Let $|W|$ denote the length of $W$.  In this section we prove the following:

\begin{theo}\label{thm:1step}
We have
$$l\label{eqn:1step}
\liminf_{\delta\to 0}\sup_{W:|W|<\delta}\sum_i\lambda_i^{-1}<1,
$$l
where the supremum is taken over all such unstable manifolds.
\end{theo}

Let $\Sigma_n$ denote those points $z\in S$ for which $H_S^2(z)=H^{n+1}(z)$. The sets $\Sigma_n$ are to $H_S^2$ what the sets $S_n$ are to $H_S$. In fact, for large $n$, $S_n\subset\Sigma_n$. This follows from Lemma~\ref{lem:isolation}, or by direct construction of $\sigma^2$. See Figure~\ref{fig:H}.

Fix a small $\eps>0$, let $B_{\eps}(p)$, $B_{\eps}(q)$ be corresponding neighbourhoods of $p$ and $q$ and let $B_{\eps}(p\cup q)$ be their union. We first show that it is sufficient to consider $W$ close to either $p$ or $q$ (recall that these are the corners of $S$ at which singularities for $H_S$ and its powers accumulate).

\begin{figure}
\subfigure[The singularity set $\sigma^2$]{\label{fig:H2}\includegraphics[width=0.49\linewidth]{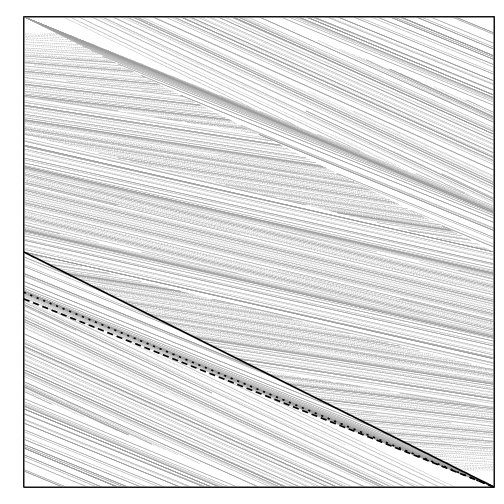}}\hfill
\subfigure[Sketch of the part of $\sigma^2$ near $p$.]{\label{fig:Hsquared}\includegraphics[width=0.49\linewidth]{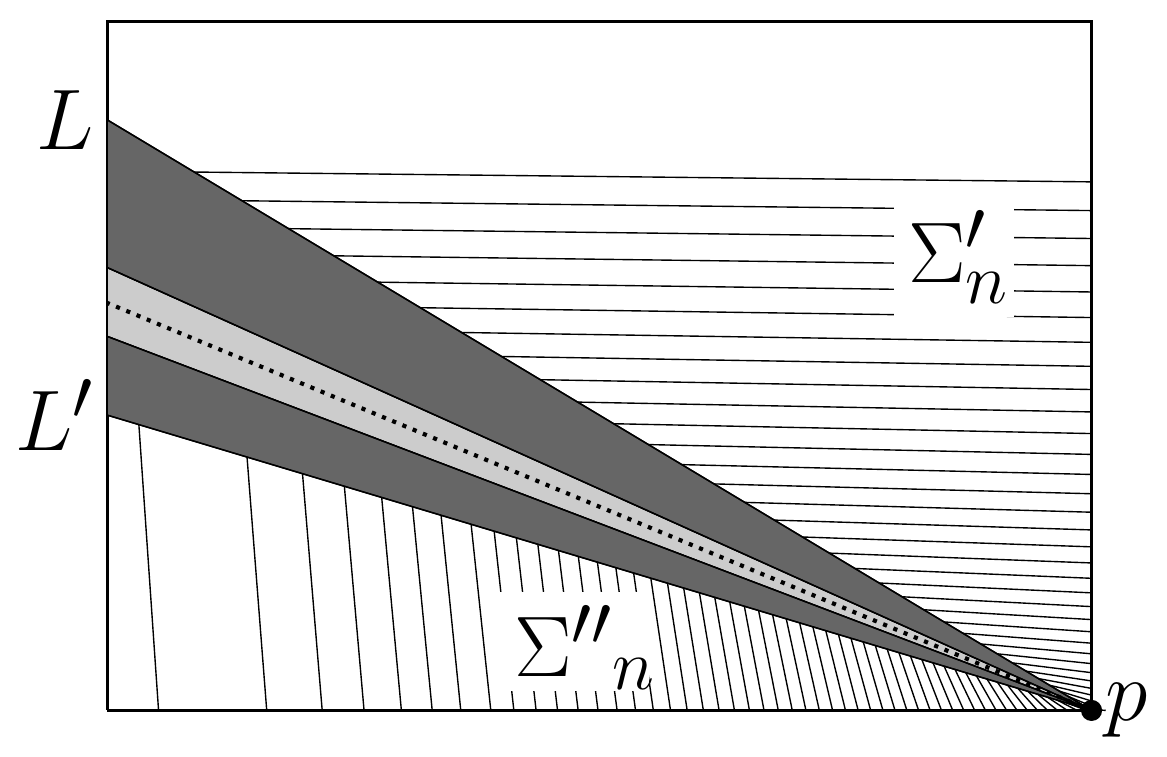}}\hfill
\caption{Illustration of $\sigma^2$, the singularity set for $H_S^2$. In (a) we show $\sigma^2$ itself, to give an idea of the overall structure, without the accumulation of components near the corners of the square. Also shown are the lines $L$ (solid line), $L'$ (dashed line) and the unstable manifold of $p$ (dotted line). In (b) we show a sketch of $\sigma^2$ in the neighbourhood of $p$. Here the regions $\bigcup \Sigma'_n$ and $\bigcup \Sigma''_n$ accumulate on $p$. The shaded region between the lines $L:y=(1-x)/2$ and $L':y=2(1-x)/5$ effectively contains a copy of $\sigma$ (see appendix~\ref{app:A}). The light shaded region surrounding the dotted line is the component of $\sigma^2$ which contains the unstable manifold of $p$.}\label{fig:H}
\end{figure}

\begin{lemma}\label{lem:nhood}
If $W\cap B_{\eps}(p\cup q)=\emptyset$ then (\ref{eqn:1step}) holds.
\end{lemma}

\begin{proof}
The singularity set $\sigma^2$ is shown in Figure~\ref{fig:H};  $\sigma^2\back B_{\eps}(p\cup q)$ consists of a finite number of line-segments. Any intersections $W\cap\sigma^2\neq\emptyset$ are transversal because $W$ and $\sigma^2$ have all tangent vectors in unstable and stable cones respectively. Thus a vanishingly short $W$ can cross at most the largest number of $\sigma^2$ line segments as meet at any point, which by inspection is two. 

Such a $W\back\sigma^2$ has at most three connected components $W_i$ so that the sum \eqref{eqn:1step} contains at most three terms. The maximum value for $\lambda_i^{-1}$ occurs in $\Sigma_1$ and is the largest eigenvalue of $(DG\cdot DF)^2$, i.e.\ $(3+2\sqrt{2})^2$. Thus $3/(3+\sqrt{2})^2<1$ is an upper bound for the one-step expansion factor.
\end{proof}
\medbreak

It remains to consider $W\cap B_{\eps}(p\cup q)\neq\emptyset$ or, without loss of generality, $W\subset B_{\eps}(p\cup q)$. Here singularity line-segments accumulate, with the consequence that vanishingly short unstable manifolds $W$ may intersect many of them.

Our next result describes the expansion factor on $\Sigma_n$ for large $n$.

\begin{lemma}\label{lem:lambda_n}
If $W_n\subset\Sigma_n$ and $\lambda_n$ is the associated local expansion factor then $\lambda_n\sim 24n$.
\end{lemma}

\begin{proof}
$\lambda_n$ is given by the largest eigenvalue of $DH_S^2$. It follows from Lemma~\ref{lem:isolation} that for large $n$ there are just four possible forms for $DH_S^2$:
$$
DG^n\cdot DF\cdot DH,\quad DG\cdot DF^n\cdot DH,\quad DH\cdot DG^n\cdot DF\and DH\cdot DG\cdot DF^n,
$$
where $DH=DG\cdot DF$. In the first case $\lambda_n$ is the largest eigenvalue of
$$
\left(\begin{array}{cc} 5 & 12 \\ 10n+2 & 24n+5 \end{array}\right),
$$
given by $12n+5+\sqrt{144n^2+120n+24}\sim 24n$. The remaining cases are similar.
\end{proof}
\medbreak

Consider an unstable manifold $W\subset B_{\eps}(p)$, having all tangent vectors in $C^+$. \emph{A priori} Lemma~\ref{lem:lambda_n} poses a problem as follows. By taking $W$ as close to vertical as is required and by insisting that $W$ intersects as small a neighbourhood of $p$ as is required, $W$ can intersect arbitrarily many sets $\Sigma_n$, making the associated one-step growth factor $\sum 1/24n$ arbitrarily large. This is clearly incompatible with \eqref{eqn:1step}.

The key to overcoming this seeming difficulty is to observe that $W$ cannot be both `close to vertical' and `close to $p$' simultaneously. The following lemma formalises this idea.

\begin{lemma}\label{lem:gradient}
Let $(u,v)^T$ be any tangent to $W\subset B_{\eps}(p\cup q)$. Then $v/u\to 1+\sqrt{2}$ as $\eps\to 0$.
\end{lemma}

\begin{proof}
The result  is another consequence of Lemma~\ref{lem:isolation}. If $W\subset B_{\eps}(p\cup q)$, $z\in W\cap\Sigma_n$ and $\eps$ is `small' then either $n=1$ or $n$ is `large'.

First suppose $n$ is large, then by Lemma~\ref{lem:isolation} at least $m\sim\kappa\ln n$, $\kappa>0$, of the immediate $H_S^{-2}$-images of $z$ are in $\Sigma_1$. Let $(u',v')^T\in C^+$ be a tangent to $H_S^{-2m}(W)$ at $H_S^{-2m}(z)=H^{-2m}(z)$. The tangent to $W$ at $z$ is then given by $(u,v)^T=(DG\cdot DF)^{2m}(u',v')^T$. As $\eps\to 0$, so $m\to\infty$, and the tangent to $W$ approaches the unstable eigenvector of $DG\cdot DF$, i.e.\ $v/u\to 1+\sqrt{2}$.

Conversely suppose $z\in W\cap\Sigma_1$. Although Lemma~\ref{lem:isolation} doesn't apply immediately, the proximity of $z$ to $p$ or $q$ has the same consequence that some number of the immediate $H_S^{-2}$-images of $z$ are in $\Sigma_1$, and from here the argument is as above.
\end{proof}
\medbreak

To establish~\eqref{eqn:1step} it is enough (Lemma~\ref{lem:nhood}) to consider $W$ approaching $p$ (or $q$, but without loss of generality we focus on the former). Such a $W$ can intersect arbitrarily many sets $\Sigma_n$ but our control over its orientation (Lemma~\ref{lem:gradient}) and knowledge of the growth factor on $\Sigma_n$ (Lemma~\ref{lem:lambda_n}) will lead to a bound on the corresponding sum $\sum\lambda_n^{-1}$. The next lemma establishes such a bound.

Consider the conical region of $B_{\eps}(p)$ bounded by $L$ (satisfying $y=(1-x)/2$) and by $x=1$. Let $\Sigma_n'$ be the component of $\Sigma_n$ in this region. We consider the connected component of $W$ confined to this region; the case of more general $W$, establishing the theorem, follows the lemma.

\begin{lemma}\label{lem:above_funnel}
Let $W$ be a connected component of an unstable manifold, of length $\delta>0$, with ends on $L$ and $x=1$. Then
$$
\liminf_{\delta\to 0}\sum_n\lambda_n^{-1}=\frac{1}{24}\ln\left(3+2\sqrt{2}\right).
$$
\end{lemma}

\begin{proof}
Let $W\cap L=\left(1-2y_0,y_0\right)$ for some small $y_0>0$. Lemma~\ref{lem:gradient} says that $W$ intersects $x=1$ at $y=y_1\approx(3+2\sqrt{2})y_0$, with equality in the limit $\delta\to 0$.

The `lower' boundary of $\Sigma_n'$, which is also the upper boundary of $\Sigma_{n+1}'$, intersects $x=1$ at $y=1/2n$ and intersects $L$ at $y=1/2(n-1)$. Let $N$ be the unique integer so that
$$
\frac{1}{2N}< y_1\leqs\frac{1}{2(N-1)}.
$$
$W$ intersects $\Sigma_N'$ but not $\Sigma_n'$ for any $n<N$. The asymptotic relationship between $y_1$ and $y_0$ says that there is an integer $M$ depending on $N$ so that $W$ intersects $\Sigma_M'$, $W$ does not intersect $\Sigma_m'$ for any $m>M$, and $M\sim(3+2\sqrt{2})N$. 

The limit $N\to\infty$ corresponds to $\delta\to 0$, thus
$$
\liminf_{\delta\to 0}\sum_n\lambda_n^{-1}=\lim_{N\to\infty}\sum_{n=N}^{M(N)}\lambda_n^{-1}=\lim_{N\to\infty}\sum_{n=N}^{\lceil(3+2\sqrt{2})N\rceil}\frac{1}{24n}=\frac{1}{24}\ln(3+2\sqrt{2}),
$$
where we have used the fact that $\sum_{n=1}^N\frac{1}{n}-\ln N\to\text{const}$.
\end{proof}
\medbreak

We now prove the main result of the section.

\begin{proof}[Proof of Theorem~\ref{thm:1step}]
In light of the comments made following Lemma~\ref{lem:gradient} we let $W$ have end-points on $y=0$ and $x=1$ and gradient $1+\sqrt{2}$. Such a $W$ intersects four distinct regions where singularities accumulate, and the region $\Sigma_1$. We determine the contribution to $\liminf_{\delta\to 0}\sup_{W:|W|<\delta}\sum_i\lambda_i^{-1}$ of each.

Lemma~\ref{lem:above_funnel} deals with the region $\bigcup\Sigma_n'$, bounded by $L$ and $x=1$. Now let $\Sigma_n''$ be the connected component of $\Sigma_n$ in the region bounded by $L'$ satisfying $y=2(1-x)/5$ and by $y=0$. Let $\tau$ denote reflection through $x+y=1$. Notice that
$$
\tau\circ F(\Sigma_n'')=\Sigma_n'
$$
for each $n$. Moreover the gradient of $W$ is $(\tau\circ F)$-invariant. Thus Lemma~\ref{lem:above_funnel} applies to $\tau\circ F\left(\bigcup\Sigma_n''\right)$ and, because $\tau\circ F$ is invertible, to $\bigcup\Sigma_n''$ itself.

The remaining two regions where singularities accumulate can be dealt with in the same manner. For the region adjacent to $\bigcup\Sigma_n'$ the appropriate invertible transformation is $G\circ F$ and for the region adjacent to $\bigcup\Sigma_n''$ it is $G\circ F\circ\tau\circ F$.

Finally, $W$ crosses $\Sigma_1$ where, by Lemma~\ref{lem:nhood},  $\lambda(z)=(3+2\sqrt{2})^2$. We conclude that
$$
\liminf_{\delta\to 0}\sup_{W:|W|<\delta}\sum_i\lambda_i^{-1}=\frac{1}{(3+2\sqrt{2})^2}+\frac{1}{6}\ln(3+2\sqrt{2})<1.
$$
\end{proof}
\medbreak

\section{Proof of the main result}\label{Proof}

The following theorem shows that the analogue of \eqref{eqn:young2} is satisfied.

\begin{theo}\label{thm:exp}
There is a set $\Lambda\subset S$ of positive $\mu_S$-measure and having hyperbolic product structure, and $\theta\in(0,1)$ so that
$$
\mu\{z\in S:\Rtn^*(z;H_S,\Lambda)>n\}=\mathcal{O}(\theta^n).
$$
\end{theo}

\begin{proof}
As described in Section~\ref{Background} it is not necessary to explicitly construct $\Lambda$. Rather we show that $H_S$ satisfies the conditions, essentially due to Chernov \cite{chernov}, that were listed. The necessary work was completed in Sections~\ref{Partition}, \ref{Bernoulli} and~\ref{OneStep}. We remind the reader where each result may be found; italics correspond to subsection headings of Section~\ref{Background}.

The \emph{smoothness} condition concerns the set $\sigma$ defined in Section~\ref{Partition}. It is evidently closed and a countable union of zero-measure line segments. Much of the \emph{hyperbolicity} condition is demonstrated in Lemma~\ref{lem:cones} with the remainder, concerning Lyapunov exponents and local invariant manifolds, in Lemmas~\ref{lem:KS} through~\ref{lem:LE}. Existence of an invariant \emph{SRB measure} is given by Proposition~\ref{prop:ergodic}. That it is mixing follows from the stronger Bernoulli property, proved in Theorem~\ref{thm:Bernoulli}. The conditions on \emph{distortion bounds} and \emph{bounded curvature} follow immediately from the piecewise linearity of $H$ and thus of $H_S$. Indeed, the expansion factor must be constant on local invariant manifolds, which themselves are zero-curvature line-segments. \emph{Absolute continuity} of the foliation follows from the result of \cite{ks} and the Lemmas~\ref{lem:KS} through~\ref{lem:LE}. The condition on the \emph{structure of the singularity set} holds because any unstable curve $W$ intersects $\sigma$ transversally and at most countably many times; the only possible accumulation points are $p$ and $q$ and the rate of convergence along such a sequence of intersections is of order $1/n$, owing to the structure of $\sigma$. The \emph{one-step growth} condition was the subject of Theorem~\ref{thm:1step}. Finally, $\mu$ and $\mu_S$ differ by a constant factor on $S$ so that the statement holds as given, i.e.\ with $\mu$ rather than $\mu_S$.
\end{proof}
\medbreak

Theorem~\ref{thm:exp} shows that $H_S:S\to S$ has exponential decay of correlations for $\alpha$-H\"{o}lder observables and constitutes the majority of our work in proving Theorem~\ref{thm:main}.

To conclude the proof we establish \eqref{eqn:young}, i.e.\ with $\Lambda$ as above and with
$$
A_n=\{z\in R:\Rtn^*(z;H,\Lambda)>n\},
$$
we show that

\begin{theo}\label{thm:poly}
$\mu(A_n)=\mathcal{O}(1/n)$.
\end{theo}

As mentioned in Section~\ref{Background} we follow a procedure introduced in \cite{markarian2}. It involves treating separately a certain set of `infrequently returning' points, to be defined. For $n\in\mathbb{N}$ and $z\in R$ let
$$
r(z;n,S)=\sum_{i=1}^n\chi_S\left(H^i(z)\right)
$$
be the number of the first $n$ images of $z$ that are in $S$. Let $b>0$ be a constant (to be fixed shortly) and define
$$
B_{n,b}=\{z\in R:r(z;n,S)>b\ln n\}.
$$
$B_{n,b}$ contains those points returning to $S$ at least $b\ln n$ times within $n$ iterations. 

\begin{lemma}\label{lem:intersection}
$\mu(A_n\cap B_{n,b})=\mathcal{O}(1/n)$.
\end{lemma}

The proof is due to \cite{markarian2,cz} but is included for the sake of completeness.

\begin{proof}
Let $z\in A_n\cap B_{n,b}$ and let $i=\Rtn(z;H,S)$. Clearly $0\leqs i<n$ (in fact $i<n-b\ln n$, but the weaker bound will suffice). From the definitions of $A_n$ and $B_{n,b}$ we have
$$
\Rtn^*\left(H^i(z);H_S,\Lambda\right)>b\ln n.
$$
and so
$$
A_n\cap B_{n,b}\subset\bigcup_{i=0}^{n-1}H^{-i}\{z'\in S:\Rtn^*(z';H_S,\Lambda)>b\ln n\}.
$$
Theorem~\ref{thm:exp} now gives
$$
\mu(A_n\cap B_{n,b})\leqs n\mu\{z'\in S:\Rtn^*(z';H_S,\Lambda)>b\ln n\}=\mathcal{O}\left(n\theta^{b\ln n}\right).
$$
Taking $b>-2/\ln\theta>0$ gives $\theta^{b\ln n}<n^{-2}$ and thus the result.
\end{proof}
\medbreak

Let us now fix $b>-2/\ln\theta>0$ as above. For $z\in R$ and $n\in\mathbb{N}$ let
$$
N_{\textup{max}}(z,n)=\max\left\{\Rtn\left(H^{i}(z);H,S\right):0\leqs i\leqs n\right\}.
$$
$N_{\textup{max}}(z,n)$ is the largest interval either until we first enter $S$, or between consecutive returns to $S$. The main step in our proof that $\mu(A_n\back B_{n,b})=\mathcal{O}(1/n)$ is to show that, away from $B_{n,b}$, $N_{\textup{max}}$ grows linearly.

\begin{lemma}\label{lem:beta}
There is a constant $\beta>0$ such that if $r(z;n,S)\leqs b\ln n$ then $N_{\textup{max}}(z,n)\geqs\beta n$.
\end{lemma}

\begin{proof}
Suppose that $r(z;n,S)\leqs b\ln n$. $N_{\textup{max}}(z,n)$ is minimised when the total return-time $n$ is distributed as evenly as possible between the $b\ln n$ returns. For small $n$ the mean return-time $n/b\ln n$ gives a lower bound.

Now suppose that
$$
\frac{n}{b\ln n}\geqs Ke^k,
$$
where $k$, $K$ are the constants of Lemma~\ref{lem:isolation}. Some returns necessarily land in $S_N$, $N\geqs Ke^k$, and so adjacent returns are to $S_1$. Thus evenly distributing the total return-time $n$ between the $b\ln n$ returns is inconsistent with Lemma~\ref{lem:isolation}. In this case $N_{\textup{max}}(z,n)$ is minimised by an itinerary of the form
$$
..., S_1, S_{N_1}, S_1, S_{N_2}, S_1, S_{N_3}, ...
$$
where each $N_i$ is approximately $2n/b\ln n-1\leqs N_{\textup{max}}(z,n)$.

If additionally
$$
\frac{2n}{b\ln n}-1\geqs Ke^{2k}
$$
then the above arrangement into pairs is also inconsistent with Lemma~\ref{lem:isolation}. Here $N_{\textup{max}}$ is minimised by an itinerary
$$
..., S_1, S_1, S_{N_1}, S_1, S_1, S_{N_2}, S_1, S_1, S_{N_3}, ...
$$
where each $N_i$ is approximately $3n/b\ln n-2\leqs N_{\textup{max}}(z,n)$.

In general if $j\in\mathbb{N}$ and
$$l\label{eqn:if}
\frac{jn}{b\ln n}-(j-1)\geqs Ke^{jk}
$$l
then
$$l\label{eqn:then}
N_{\textup{max}}(z,n)\geqs\frac{(j+1)n}{b\ln n}-j.
$$l
For $n\in\mathbb{N}$ let $J(n)$ be the largest integer $j$ for which \eqref{eqn:if} holds, then 
$$
\frac{(J+1)n}{b\ln n}-J< Ke^{(J+1)k}.
$$
It follows that
$$l\label{eqn:LB_j}
\frac{n}{b\ln n}<\frac{1}{J+1}\left(Ke^{(J+1)k}+J\right)<Ke^{(J+1)k}.
$$l
The first inequality in \eqref{eqn:LB_j} is a rearrangement of the previous displayed equation. The second follows easily from the assumption that $Ke^{(J+1)k}>1$, which holds for all sufficiently large $J$, i.e.\ for all sufficiently large $n$.

Taking logarithms on each side of \eqref{eqn:LB_j} and rearranging gives
$$
J+1>\frac{1}{k}\left(\ln n-\ln\ln n-\kappa\right)
$$
where $\kappa=\ln K+\ln b$ is constant. By assumption \eqref{eqn:if} holds with $j=J$ therefore \eqref{eqn:then} gives
$$
N_{\textup{max}}(z,n)>\frac{n}{kb}\left(1-\frac{\ln\ln n}{\ln n}-\frac{\kappa}{\ln n}\right)-\frac{1}{k}\left(\ln n-\ln\ln n-\kappa\right)+1.
$$
Clearly $N_{\textup{max}}/n\to 1/kb>0$ as $n\to\infty$, at a rate that is independent of $z$. Hence the lemma holds for some $0<\beta<1/kb$.
\end{proof}
\medbreak

\begin{lemma}\label{lem:complement}
$\mu(A_n\back B_{n,b})=\mathcal{O}(1/n)$.
\end{lemma}

\begin{proof}
We prove the sufficient result $\mu(R\back B_{n,b})=\mathcal{O}(1/n)$. By Lemma~\ref{lem:beta}
\begin{eqnarray*}
\displaystyle\mu\left(R\back B_{n,b}\right)&=&\mu\left\{z\in R:r(z;n,S)\leqs b\ln n\right\} \\
&\leqs& \displaystyle\mu\left\{z\in R:N_{\textup{max}}(z,n)\geqs\beta n\right\}.
\end{eqnarray*}

Recall that $N_{\textup{max}}(z,n)$ is either the largest $N$ such that $z$ lands in $S_N$, or the number of iterations taken to first enter $S$. If it is the former then $H^i(z)\in\{S_N:N\geqs\beta n\}$ for some $0\leqs i\leqs n$. We have seen, in the proof of Lemma~\ref{lem:OS}, that $\mu(S_n)=\mathcal{O}(1/n^3)$, and thus $\mu\{S_N:|N|\geqs\beta n\}=\mathcal{O}(1/n^2)$, and thus
$$
\mu\left(R\back B_{n,b}\right)\leqs\mu\left(\bigcup_{i=0}^nH^{-i}\{S_N:|N|\geqs\beta n\}\right)=\mathcal{O}(1/n).
$$

Conversely, notice that $z\notin S$ and let $H^j(z)\in S$ be the smallest such $j$. For $n\in\mathbb{N}$ let $S_{-n}$ be the elements of the partition of $S$ induced by $H_S^{-1}$ (these are analogous to the sets $S_n$ induced by $H_S$). Then $H^j(z)\in S_{-M}$ and $M>N_{\textup{max}}(z,n)\geq\beta n$, and the argument proceeds as above.
\end{proof}
\medbreak

Theorem~\ref{thm:poly} follows immediately from Lemmas~\ref{lem:intersection} and Lemmas~\ref{lem:complement}. This completes the proof of Theorem~\ref{thm:main}.

Finally we remark on the changes required to our proof in order to accommodate different annuli $P$ and $Q$. The differences most obviously manifest themselves in the values of certain constants. Lyapunov exponents in particular will vary, impacting the various constants in Sections~\ref{Partition} and~\ref{Bernoulli}. The exact values however are unimportant to the arguments, and important properties, e.g.\ positivity, will not change. Geometric features such as the precise structure of $\sigma$ and the relative sizes of the sets $S_n$ will also change, but again the important features, such as the general structure of $\sigma$ and the bound $\mu(S_n)=\mathcal{O}(1/n^3)$, remain. Lastly, if local growth rates are weaker than in the map considered, we might need to consider a higher iterate of $H_S$ in Section~\ref{OneStep}, but this introduces no further difficulties than have presently been overcome.

\section{Correlation decay in other linked-twist maps}\label{Others}

We have shown that the rate of decay of correlations for a large class of linked twist maps is polynomial. However, note that if both of the annuli are thickened until they are equal to the entire torus, the map $H$ becomes the hyperbolic toral automorphism known as the Arnold Cat Map, which is well-known to be exponentially mixing. The transition from non-uniformly hyperbolic linked-twist map to uniformly hyperbolic Cat Map deserves further study.

Similarly, increasing the {\em wrapping number} of the twists (that is, taking at least one of $j,k>1$) only serves to enhance the mixing, yet the behaviour at the boundary remains linear. In this case we expect only minor modifications to produce identical results. An interesting case arises if the twists $f$ and $g$ are permitted to be nonlinear, yet still monotonic. As in \cite{be,p1}, $H$ is still Bernoulli, although the behaviour at the boundaries may now be different. Again, we expect the dominant behaviour to be sub-exponential in this situation.

If exactly one of $j$ and $k$ is allowed to be a negative integer, the situation is far more complicated. Although in this case $H$ can still be shown to be Bernoulli, for certain choices of $f$ and $g$, the proof relies on an intricate geometrical argument due to \cite{p1}. The question of its rate of correlation decay is still open. Likewise the LTMs defined on planar annuli of \cite{woj} are Bernoulli \cite{springham}, but the methods in this article would need significant adaption.

\section*{Acknowledgments}
It is our pleasure to acknowledge the financial support of Leverhulme Trust grant number F/10101/A. We thank Ian Melbourne for much encouragement and advice and also Stefano Luzzatto, Matthew Nicol and Stephen Wiggins for helpful discussions. Some of this work was completed whilst the authors were guests of the Institute of Mathematics and its Applications at the University of Minnesota and we are grateful for their hospitality. Further work was completed whilst JS was a visitor at the University of Bristol, to whom he is also grateful.

\appendix

\section{Structure of the singularity sets}\label{app:A}

In this appendix we give details of the construction of the singularity sets for $F_S$, $G_S$, $H_S$ and $H^2_S$. Each set consists of the pre-images of the boundary of $S$ under the return map in question, and which partition $S$ into distinct regions which take different numbers of iterates to return to $S$ under the map ({\em not} the return map) in question. 

For example, the lower half of the singularity set for $F_S$ consists of a sets of lines connecting $2y=1-x$ with $x=1$. Each line in this set meets $2y = 1-x$ at $((n - 2)/(n - 1), 1/2(n - 1))$ and meets $x = 1$ at $(1, 1/2n)$, for each $n \ge 2$. These lines are just the lines which take $n$ iterates of $F$ to be mapped into the line $x = 0$, and hence to return to $S$. Similarly, the upper half of this singularity set contains a set of lines connecting $2y=2-x$ with $x=0$, which are the lines which takes $n$ iterates to be mapped into $x=1$. These sets of lines accummulate on $q$ and $p$ respectively. See figure~\ref{fig:part_1}.

The singularity set of $G_S$, shown in figure~\ref{fig:part_2}, is exactly analogous. This set is just the singularity set for $F_S$ reflected about the line $y=1-x$. Again, constituent lines accummulate on $p$ and $q$. 

\begin{figure}
\subfigure[$S^{(F)}_1$]{\label{fig:sing1}\includegraphics[width=0.24\linewidth]{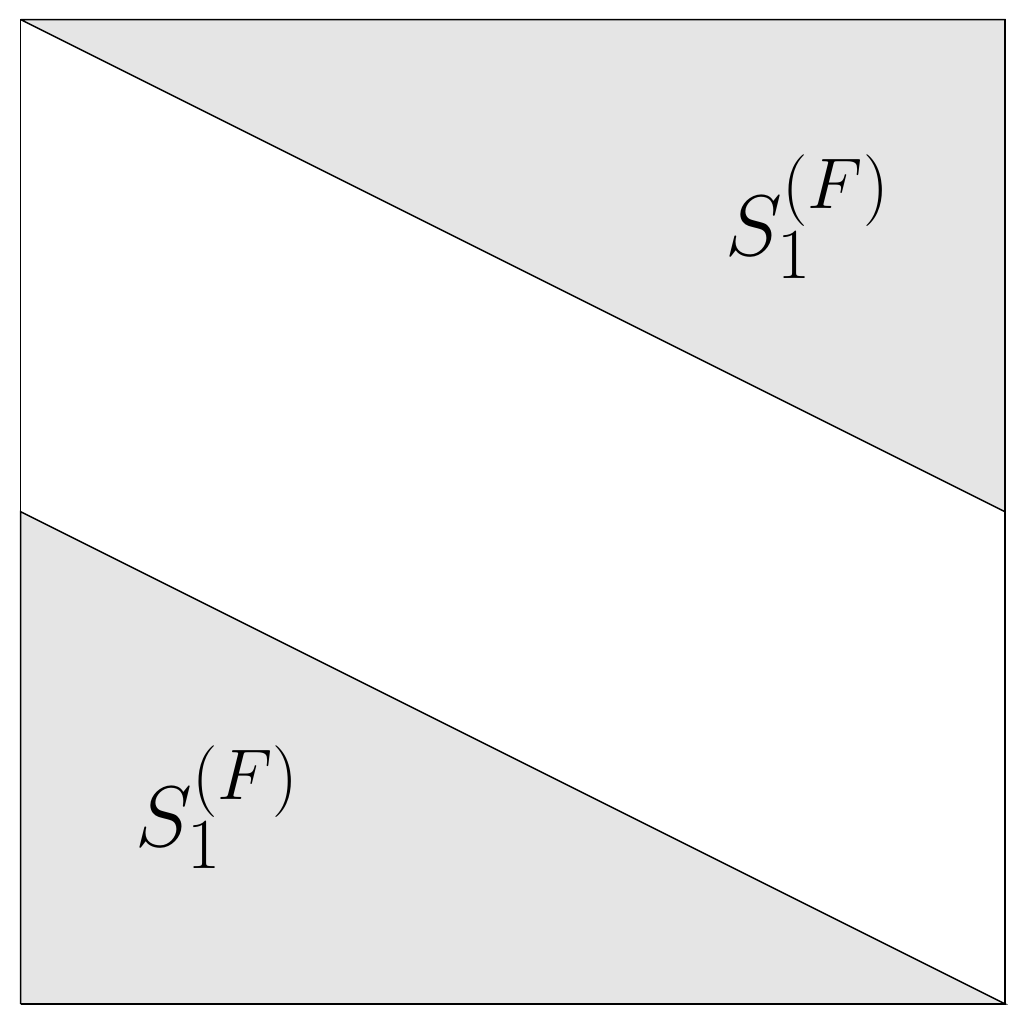}}
\subfigure[$S^{(G)}_1$]{\label{fig:sing2}\includegraphics[width=0.24\linewidth]{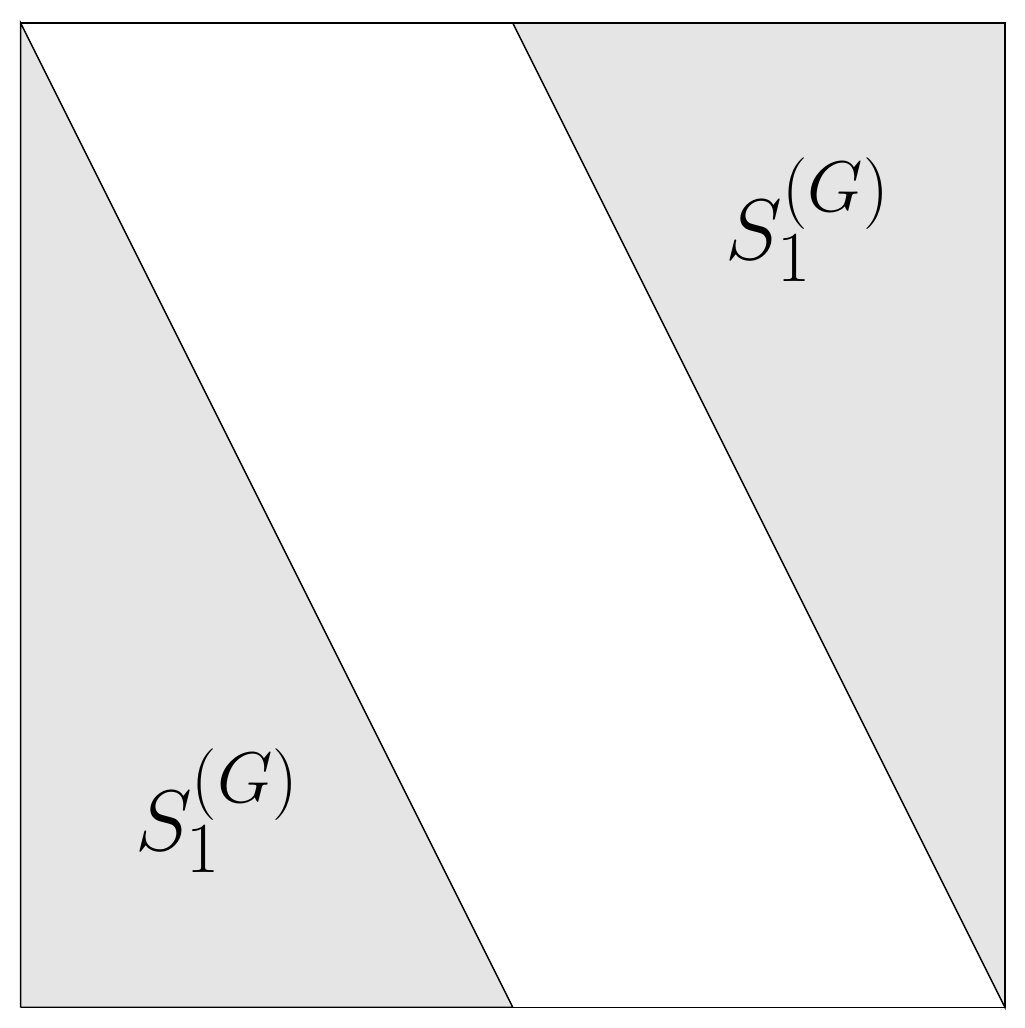}}
\subfigure[$F^{-1}(S^{(G)}_1)$]{\label{fig:sing3}\includegraphics[width=0.24\linewidth]{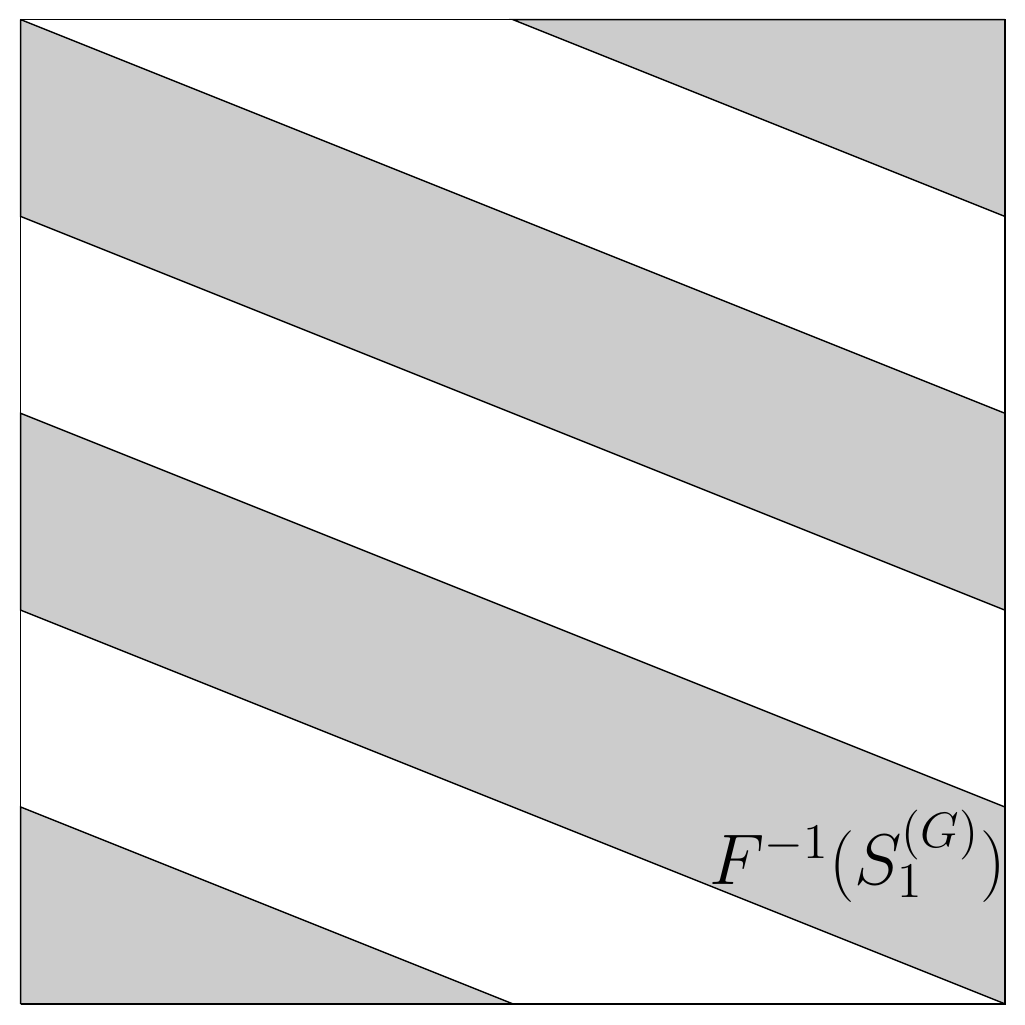}}
\subfigure[$S^{(F)}_1 \cap F^{-1}(S^{(G)}_1)$]{\label{fig:sing5}\includegraphics[width=0.24\linewidth]{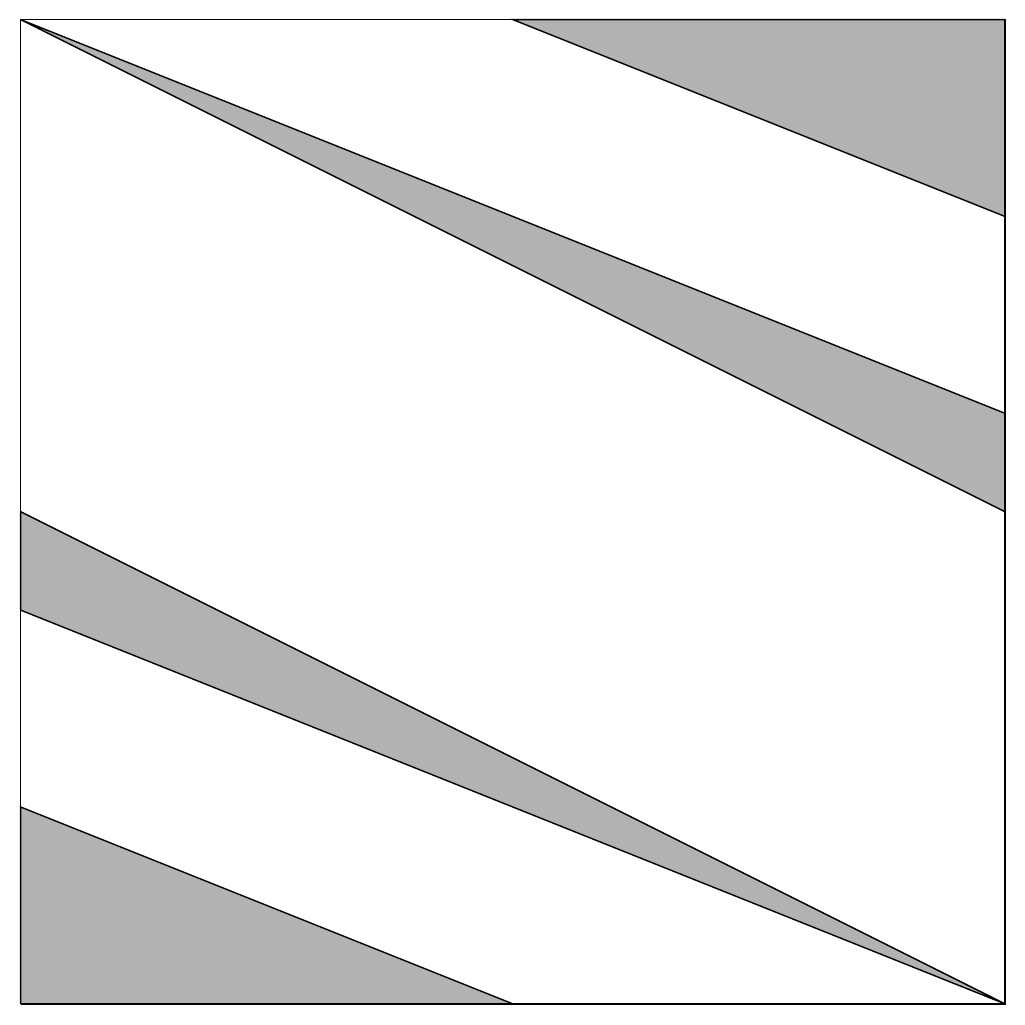}}\hfill
\caption{The construction of the set $S_1$, consisting of all points of $S$ which return to $S$ under one iterate of $H$. Sets $S_i$ for larger $i$ can be considered in the same way, and their union forms the partition in figure~\ref{fig:part_3}.}\label{fig:sing}
\end{figure}

The singularity set for $H_S$, denoted $\sigma$, can be constructed in a similar way. We are also interested in the sets $S_n=\left\{z\in S:\Rtn(z;H,S)=n\right\}$ so, analogously, we define the sets
\begin{eqnarray*}
S^{(F)}_n&=&\left\{z\in S:\Rtn(z;F,S)=n\right\}\\
S^{(G)}_n&=&\left\{z\in S:\Rtn(z;G,S)=n\right\}
\end{eqnarray*}
These sets can be easily discerned from figures~\ref{fig:part_1} and \ref{fig:part_2}, with $S^{(F)}_n$ and $S^{(G)}_n$ accumulating on $p$ and $q$ as $n \to \infty$. Now by definition of $H_S$, the sets $S_n$ have the property that a point $z \in S_n$ if and only if $z \in S^{(F)}_j \cap F^{-j}(S^{(G)}_k)$, where $j+k = n+1$. We illustrate this statement by constructing $S_1$ explicitly, as shown in figure~\ref{fig:sing}. Points in the shaded region of figure~\ref{fig:sing1} return to $S$ under $j=1$ iterates of $F$. Points in the shaded region of figure~\ref{fig:sing2} return to $S$ under $k=1$ iterates of $G$. Thus to return to $S$ under a single iterate of $H$, a point must lie in both $S_1^{(F)}$, and the pre-image under $F$ of $S_1^{(G)}$ (shown in figure~\ref{fig:sing3}). This intersection is shaded in figure~\ref{fig:sing5}. 

A similar statement can be made for $\sigma^2$. A point $z \in \Sigma_n$ if and only if 
\begin{equation}\label{eq:app1}
z \in S^{(F)}_{j_1} \cap F^{-j_1}(S^{(G)}_{k_1}) \cap F^{-j_1}(G^{-k_1}(S^{(F)}_{j_2})) \cap F^{-j_1}(G^{-k_1}(F^{-j_2} (S^{(G)}_{k_2}))),
\end{equation}
where $j_1 + k_1 + j_2 + k_2 = n+3$. Thus the statement that the shaded region of figure~\ref{fig:H} effectively contains a copy of $\sigma$ can be understood by the fact that this shaded region is part of $S_1$, and equation (\ref{eq:app1}) must have $j_1 = k_1 = 1$, leaving $j_2$ and $k_2$ to satisfy $j_2+k_2 = n+1$, just as in the construction of $\sigma$.

\bibliographystyle{abbrv}
\bibliography{JamesBib}

\end{document}